\documentclass[11pt]{article}
\usepackage[ruled,vlined,linesnumbered]{algorithm2e}
\usepackage{amsmath,amsfonts,amssymb,amsthm,latexsym}
\usepackage{epsfig}
\usepackage{color}
\usepackage{enumerate}
\usepackage{hyperref}
\usepackage{enumerate}
\usepackage{graphicx}
\usepackage[dvipsnames]{xcolor}
\usepackage[cp1250]{inputenc}
\usepackage{tikz}
\usepackage{tkz-graph}
\usepackage{tkz-berge}
\usepackage{subcaption}

\newcommand{\smallqed}{{\tiny ($\Box$)}}

\newcommand{\fk}{\chi_{t}^{f}}
\newcommand{\fd}{\chi_{2}^{f}}

\def\cp{\,\square\,}

\usetikzlibrary{arrows.meta,shapes.arrows}
\hypersetup{colorlinks=true, linkcolor=blue, citecolor=blue,
urlcolor=blue}
\usetikzlibrary{calc}

\hypersetup{colorlinks=true}

\hypersetup{colorlinks=true, linkcolor=blue, citecolor=blue,urlcolor=blue}

\def\cp{\,\square\,}

\hypersetup{colorlinks=true}

\hypersetup{colorlinks=true, linkcolor=blue, citecolor=blue,urlcolor=blue}

\title{Frugal coloring of graphs revisited}
\date{}
\author{{Bo\v{s}tjan Bre\v{s}ar$^{a,b}$, Wenjie Hu$^{c,a}$ and Babak Samadi$^{b}$}\vspace{2mm}\\
$^{a}$Faculty of Natural Sciences and Mathematics, University of Maribor, Slovenia\vspace{1mm}\\
$^{b}$Institute of Mathematics, Physics and Mechanics, Ljubljana, Slovenia\vspace{1mm}\\
$^{c}$Hubei Key Laboratory of Applied Mathematics, Faculty of Mathematics and Statistics\\
Hubei University, Wuhan 430062, China\vspace{1.5mm}\\
\texttt{bostjan.bresar@um.si}\\
\texttt{huwj@stu.hubu.edu.cn}\\
\texttt{babak.samadi@imfm.si}}
\date{}

\newtheorem{theorem}{Theorem}[section]
\newtheorem{corollary}[theorem]{Corollary}
\newtheorem{lemma}[theorem]{Lemma}
\newtheorem{observation}[theorem]{Observation}
\newtheorem{proposition}[theorem]{Proposition}

\newtheorem{p}{Problem}

\theoremstyle{definition}

\theoremstyle{remark}

\begin{document}

\maketitle

\begin{abstract}
Given a graph $G$ and a positive integer $t$, an independent set $S\subseteq V(G)$ is $t$-frugal if every vertex in $G$ has at most $t$ vertices from $S$ in its neighborhood. A $t$-frugal coloring of $G$ is a partition of its vertex set into $t$-frugal independent sets. The maximum cardinality of a $t$-frugal independent set in $G$ is denoted by $\alpha_t^f(G)$, while the minimum cardinality of a $t$-frugal coloring of $G$, $\chi_t^f(G)$, is called the $t$-frugal chromatic number of $G$. Frugal colorings were introduced in 1998 by Hind, Molloy and Reed, and studied later in just a handful of papers. In this paper, we revisit this concept by studying it from various perspectives. While the NP-hardness of frugal coloring is known, we prove that the decision version of $\alpha_t^f$ is NP-complete for any positive integer $t$ even when restricted to bipartite graphs, and present a linear-time algorithm to determine its value for trees. 
We present several bounds on both parameters. In particular, for any positive integer $t$, we prove a general sharp lower bound on $\chi_{t}^{f}(G)$ expressed in terms of $\alpha_{t}^{f}(G)$ and size of $G$. We also prove a sharp upper bound on the $2$-frugal independence number of any graph $G$, which in the case of graphs $G$ with minimum degree $\delta$ at least $2$ simplifies to $\alpha_2^f(G)\le 2n/(\delta+2)$. While a greedy upper bound for subcubic graphs $G$ yields $\chi_2^f(G)\le 7$, we obtain a substantial improvement by proving that $3\le\chi_2^f(G)\le 5$ holds for any graph $G$ with $\Delta(G)=3$. For several classes of graphs such as claw-free cubic graphs and block graphs, as well as for the Cartesian and strong products of multiple two-way infinite paths, we are able to determine the exact values of their $2$-frugal chromatic numbers. We provide sharp upper bounds for the $2$-frugal chromatic numbers in all four standard graph products, which are expressed as different invariants of their factors depending on the type of product. In the cases of Cartesian and lexicographic products, we also obtain sharp lower bounds. Finally, we obtain Nordhaus-Gaddum type results, which bound the sum of the $2$-frugal chromatic numbers of $G$ and its complement $\overline{G}$ from below and from above by functions of the order of $G$. For the upper bound  $\chi_{2}^{f}(G)+\chi_{2}^{f}(\overline{G})\leq 3n/2$, we characterize the family of extremal graphs $G$.
\end{abstract}



\section{Introduction} 

A proper vertex coloring of a graph $G$ is $t$-\textit{frugal} if no color appears more than $t$ times in the neighborhood of any vertex in $G$. The $t$-\textit{frugal chromatic number} of $G$ is the minimum $k$ for which $G$ admits a $t$-frugal coloring with $k$ colors, and we denote it by $\fk(G)$. Hind, Molloy and Reed~\cite{HMR} initiated the concept of frugal coloring in 1997 proving that a graph $G$ with sufficiently large maximum degree $\Delta$ admits a $({\log}^8{\Delta})$-frugal coloring with $\Delta+1$ colors and mentioning an application for total colorings. The result was improved by Ndreca, Procacci and Scoppola~\cite{nps} in 2012, while only a few other authors considered this type of coloring so far~\cite{AEv,BMR,KM}; see also~\cite{ch-2011} where its list version was considered. In this paper, we expand the consideration of frugal colorings by studying several aspects that are of relevance for coloring invariants. 

Frugal colorings are related to several known coloring invariants. A {\em $2$-distance coloring} of a graph $G$ is a mapping $c:V(G)\to \{1,\ldots,k\}$ such that any two vertices at distance at most $2$ receive different colors. The minimum number of colors in such a coloring is the {\em $2$-distance chromatic number} $\chi_{2}(G)$ of $G$; see Kramer and Kramer~\cite{kk0} for systematic treatment of $d$-distance coloring. Alternatively, $2$-distance coloring of $G$ coincides with the coloring of the square $G^2$ of $G$, where two vertices in $G^2$ are adjacent if they are adjacent or have a common neighbor in $G$, and so $\chi_{2}(G)=\chi(G^{2})$. Clearly, the condition of $2$-frugal coloring is weaker than that of $2$-distance coloring, thus $\fd(G)\le \chi_2(G)$ for any graph $G$. A proper coloring of a graph $G$ in which the vertices of any two color classes induce a forest of paths is a {\em linear coloring} of $G$, as introduced by Yuster~\cite{yus}. Note that a linear coloring is a $2$-frugal coloring of $G$, but the converse is not necessarily true. Thus, for the resulting invariant ${\rm lc}(G)$ of $G$, which is the minimum number of colors in a linear coloring in $G$, we get $\fd(G)\le {\rm lc}(G)$. 

Another concept related to frugal coloring arises from that of limited packing as introduced by Gallant et al.~\cite{GGHR} and studied further in~\cite{GZ,BS}. A set $B\subseteq V(G)$ is a {\em $t$-limited packing} if the closed neighborhood of each vertex in $G$ contains at most $t$ vertices in $B$. Note that when $t=1$, the resulting $1$-limited packing coincides with the concept of $2$-packing, as introduced by Meir and Moon~\cite{MM} and studied in many papers concerning graph domination. Recently, a \textit{$t$-limited packing partition} of $G$, which is a partition of $V(G)$ into $t$-limited packings, was considered and the following coloring invariant was introduced. 
The \textit{$t$-limited packing partition number}, denoted by $\chi_{\times t}(G)$, is the minimum cardinality of a $t$-limited packing partition in $G$; see~\cite{AS,AL}. 
Note that when $t\ge 2$, the partition is not necessarily a proper coloring. However, additionally imposing that the color classes in a $t$-limited packing partition are independent results precisely in a $t$-frugal coloring. In particular, $\fk(G)\ge \chi_{\times t}(G)$ holds for any graph $G$ and any positive integer $t$.

The behavior of a particular graph coloring is inherently related to the nature and structure of its color classes. On the other hand, the color classes of many types of graph colorings have been studied independently. For example, a large number of papers have been published about ``independent", ``$2$-packing", ``open packing" and ``dissociation" sets, which are color classes of the ``standard", ``$2$-distance", ``injective" and ``defective" colorings, respectively. Moreover, the color classes of frugal colorings are a variant of limited packings, notably, they are independent limited packings. Due to the fundamental role of these sets in the study of frugal colorings, we investigate them under the name of {\em frugal independent sets}. More specifically, such a set is a {\em $t$-frugal independent set} ($t$FI set for short), where $t$ is a positive integer. The maximum cardinality of a $t$FI set in $G$ will be denoted by $\alpha_t^f(G)$, and called $t$-{\em frugal independence number} of $G$.  
In line with the above arguments, we will investigate frugal independent sets along with frugal colorings both from computational and combinatorial points of view.

\subsection{Preliminaries}

Throughout the paper, we consider $G$ as a simple graph with vertex set $V(G)$ and edge set $E(G)$. In addition, $G$ is finite unless explicitly stated otherwise. We use~\cite{West} as a reference for terminology and notation which are not explicitly defined here. The (\textit{open}) {\em neighborhood} of a vertex $v$ is denoted by $N_{G}(v)$, and its {\em closed neighborhood} is $N_{G}[v]=N_{G}(v)\cup \{v\}$. The {\em minimum} and {\em maximum degrees} of $G$ are denoted by $\delta(G)$ and $\Delta(G)$, respectively. Given subsets $A,B\subseteq V(G)$, let $[A,B]$ denote the set of all edges with one endvertex in $A$ and the other in $B$. For simplicity, we use the notation $[k]$ instead of $\{1,\ldots,k\}$ for any positive integer $k$. By a $\chi_{t}^{f}(G)$-coloring we mean a $t$-frugal coloring of $G$ with $\chi_{t}^{f}(G)$ colors.

Since $\chi_{t}^{f}(G)=\chi(G)$ for each integer $t\ge \Delta(G)$, where $\chi(G)$ is the chromatic number of $G$,  we restrict our attention to the cases when $t\in[\Delta(G)-1]$. Moreover, the following inequality chain follows from the definitions. 

\begin{observation}\label{Chain}
For any graph $G$ with maximum degree $\Delta$ and positive integer $k\in[\Delta]$,
\begin{center}
$\chi(G)=\chi_{\Delta}^{f}(G)\leq \ldots\leq \chi_{2}^{f}(G)\leq \chi_{1}^{f}(G)=\chi(G^{2})$.
\end{center}
\end{observation}

Let $\mathcal{J}$ be the color classes of a $\chi_{t}^{f}(G)$-coloring, and let $u$ be a vertex in $G$ of maximum degree. Let $u\in J$, where $J\in \mathcal{J}$. By definition, $J$ is an independent set and $u$ has at most $t$ neighbors in each color class in $\mathcal{J}\setminus \{J\}$. This shows that
\begin{equation}\label{Delta}
\Delta(G)=\deg_{G}(u)=\sum_{J'\in \mathcal{J}\setminus \{J\}}|N_{G}(u)\cap J'|\leq t|\mathcal{J}\setminus \{J\}|=t\fk(G)-t.
\end{equation}
Therefore, $\fk(G)\geq \lceil \Delta(G)/t\rceil+1$. The following lower bound on the $t$-frugal chromatic number follows directly from the latter inequality and Observation \ref{Chain}. 

\begin{observation}
\label{ob:lowerboundfor2}
If $G$ is a graph and $t\ge1$, then $\fk(G)\ge \max\{\chi(G),\big\lceil\frac{\Delta(G)}{t}\big\rceil+1\}.$   
\end{observation}

Concerning upper bounds with respect to the maximum degree of a graph, we obtain the following observation, which immediately follows by using a greedy coloring algorithm.

\begin{observation}
\label{ob:boundDelta}
If $G$ is a graph with maximum degree $\Delta$, then $\chi_{t}^{f}(G)\le 1+\Delta(1+\big \lfloor\frac{\Delta-1}{t}\big\rfloor)$.     
\end{observation}
The bound in Observation~\ref{ob:boundDelta} can be sharp. For instance, when $t=2$ and $\Delta=2$, the bound reads $\chi_{2}^{f}(G)\le 1+2(1+\big \lfloor\frac{2-1}{2}\big\rfloor)=3$, and $\chi_{2}^{f}(C_{2r+1})=3$ where $r\in\mathbb{N}$.  


\subsection{Main results and organization of the paper}

We start with computational aspects of the two main invariants studied in this paper. As proved in~\cite{BMR}, the decision version of the $t$-frugal chromatic number is NP-complete for any positive integer $t$. Hence, in Section~\ref{sec:computational}, we concentrate on computational aspects with respect to the $t$-frugal independence number and prove that the decision version of $\alpha_t^f(G)$ is NP-complete for any positive integer $t$ even if restricted to bipartite graphs $G$. In contrast, we present a linear-time algorithm for computing $\alpha_t^f(T)$ in an arbitrary tree $T$ and for any $t$. 

In Section~\ref{sec:generalbounds}, we prove several general bounds involving one or both graph invariants. In particular, we prove that $\chi_{t}^{f}(G)\geq {1}/{2}+\sqrt{{1}/{4}+{2m}/({t\alpha_{t}^{f}(G))}}$ holds for any positive integer $t$ and any graph $G$ of size $m$, and we characterize the graphs that attain this bound. On the other hand, when $G$ is triangle-free, we prove an upper bound on $\chi_{2}^{f}(G)$ expressed in terms of the order of $G$ and its $2$-frugal independence number. In addition, we prove an upper bound on the $2$-frugal independence number of an arbitrary graph $G$ expressed in terms of the order of $G$ and several other invariants that depend on its pendant vertices. In Section~\ref{sec:subcubic} we focus on subcubic graphs, where our main result is that any graph $G$ with $\Delta(G)=3$ satisfies $3\le \fd(G)\le 5$. While the lower bound is trivial, the upper bound is proved by an extensive case analysis. If, in addition, $G$ is claw-free, then the equality $\fd(G)=3$ holds.

Section~\ref{sec:NG} is devoted to Nordhaus-Gaddum results, where we again focus on the $t$-frugal chromatic number where $t=2$. We prove that
$$\dfrac{n}{2}+2\le\chi_{2}^{f}(G)+\chi_{2}^{f}(\overline{G})\leq \dfrac{3n}{2},$$
where the lower bound holds for all graphs $G$ of order $n$ with the exception of $6$ sporadic graphs, while the upper holds for all graphs $G$. In addition, the upper bound holds with equality if and only if $n$ is even and $G\in \{K_{1,n-1},\overline{K_{1,n-1}}\}$.

In Section~\ref{sec:classes}, we consider several graph classes and graph operations. For an arbitrary block graph $G$ we obtain the exact value of the $2$-frugal chromatic number, that is,
 $\fd(G)=\max\big\{\omega,\big\lceil\frac{\Delta}{2}\big\rceil+1\big\}$, where $\omega$ is the clique number and $\Delta$ the maximum degree of $G$. If $G\cp H$ is the Cartesian product of graphs $G$ and $H$ we obtain the following lower and upper bound on its $2$-frugal chromatic number: $\max\{\chi_{2}^{f}(G),\chi_{2}^{f}(H)\}\le \chi_{2}^{f}(G\cp H)\le \max\{\chi_2(G),\chi_2(H)\}$. The bounds are sharp and in some cases coincide. We also obtain sharp upper bounds for the $2$-frugal chromatic numbers in strong and direct products of two graphs. In addition, we bound the lexicographic product of two graphs as follows: $\chi_{2}^{f}(H)+\Big{\lceil}\frac{\Delta(G)|V(H)|}{2}\Big{\rceil}\leq \chi_{2}^{f}(G\circ H)\leq \chi_{2}^{f}(G)|V(H)|$, and provide examples of sharpness. We continue with the quest for classes of graphs that achieve the trivial lower bound for their $2$-frugal chromatic number, notably when $\fd(G)=\lceil\Delta(G)/2\rceil+1$. We prove that several well-known infinite lattices satisfy this equality. In addition, the equality holds for Cartesian products of several (infinite) paths as well as strong products of several (infinite) paths. We conclude the paper with some open problems and directions for future research.


\section{Computational complexity}
\label{sec:computational}

Concerning the coloring parameter studied in this paper, we recall the result of Bard, MacGillivray and Redlin~\cite{BMR} from 2021. They proved that given a graph $G$ one can determine in polynomial time whether $\chi_1^f(G)\le 3$ and whether $\fk(G)\le 2$ for any positive integer $t$. On the other hand, if $t\ge 2$, it is NP-complete to determine whether $\fk(G)\le 3$, and it is NP-complete to determine whether $\chi_t^f(G)\le \ell$ when $\ell\ge 4$ and $t\ge 1$.

In what follows, we consider the computational complexity aspect of the $t$FI set problem. More formally, we analyze the following decision problem.

\begin{equation*}\label{NPCP1}
\begin{tabular}{|l|}
\hline
\mbox{{\sc $t$-Frugal Independent Set}} \\
\mbox{{\sc Instance}: A graph $G$ and an integer $k\leq|V(G)|$.}\\
\mbox{{\sc Question}: Does $G$ have a $t$FI set of cardinality at least $k$?}\\
\hline
\end{tabular}
\end{equation*}

In the following proof, we will use a reduction from {\sc Exact Cover By 3-Sets} (X3C) problem, which is defined as follows. Given a set $X$, with $|X|=3q$, and a collection $C$ of $3$-element subsets of $X$, can we find a subcollection $C'$ of $C$ such that every element of $X$ occurs in exactly one member of $C'$? If such a subcollection $C'$ exists, it is called an exact cover of $X$. It is well known that X3C problem is NP-complete~\cite{gj}.

\begin{theorem}\label{NPC1}
For each positive integer $t$, {\sc $t$-Frugal Independent Set} problem is NP-complete even for bipartite graphs.
\end{theorem}
\begin{proof}
It is known from \cite{EGM} that {\sc $1$-Frugal Independent Set}, under the name of {\sc Distance-$3$ Independent Set}, is NP-complete for bipartite graphs (by making use of a reduction from INDEPENDENT SET). In view of this, we restrict our attention to $t\geq2$. 
The problem clearly belongs to NP because checking that a given subset of vertices is indeed a $t$FI set of cardinality at least $k$ can be done in polynomial time.

We construct a reduction from X3C to our problem as follows. Let $X=\{x_{1},\ldots,x_{3q}\}$ and $C=\{C_{1},\ldots,C_{p}\}$ be an instance of X$3$C. Corresponding to each $3$-element set $C_{j}$, we associate a path $a_{j}b_{j}c_{j}$. For each element $x_{j}$, we consider a star $K^{j}_{1,t-1}$ with central vertex $x_{j}'$ and set of leaves $\{x_{j,1},\ldots,x_{j,t-1}\}$. Now, let $G$ be constructed from the above disjoint union of graphs by adding edges $x_{j}'c_{r}$ if the element $x_{j}$ is in $C_{r}$. It is easy to see that the graph $G$ is bipartite and its construction can be accomplished in polynomial time. Moreover, we set $k=p+(3t-2)q$.

Let $B$ be an $\alpha_{t}^{f}(G)$-set of cardinality at least $k$. We take any vertex $x_{j}'$. If $x_{j}'\in B$, then $(B\setminus \{x_{j}'\})\cup \{x_{j,1}\}$ is also an $\alpha_{t}^{f}(G)$-set. (In such a case, $x_{j,1}$ is the only pendant vertex adjacent to $x_{j}'$, for otherwise one can obtain a larger $t$FI set by replacing $x_{j}'$ with the pendant vertices adjacent to $x_{j}'$ in $B$, which is impossible.) Assume now that $x_{j}'\notin B$. It is readily seen that $N(x_{j}')\cap B$ must have at least $t-1$ vertices. With this in mind, let $B'$ consist of any $|\{x_{j,1},\ldots,x_{j,t-1}\}\setminus B|$ vertices of $(N(x_{j}')\setminus \{x_{j,1},\ldots,x_{j,t-1}\})\cap B$. In such a situation, $(B\setminus B')\cup \{x_{j,1},\ldots,x_{j,t-1}\}$ is an $\alpha_{t}^{f}(G)$-set containing all pendant vertices adjacent to $x_{j}'$. On the other hand, $|B\cap \{a_{r},b_{r},c_{r}\}|\geq1$ for each $r\in[p]$, since $B$ is a maximum $t$FI set in $G$. So, without loss of generality, we may assume that $a_{r}\in B$ for each $r\in[p]$. In fact, we can assume that $\{a_{1},\ldots,a_{p}\}\bigcup(\bigcup_{j=1}^{3q}\{x_{j,1},\ldots,x_{j,t-1}\})\subseteq B$. 

Note that every vertex in $B\cap \{c_{1},\ldots,c_{p}\}$ has precisely three neighbors in $\{x_{1}',\ldots,x_{3q}'\}$. Moreover, since $\bigcup_{j=1}^{3q}\{x_{j,1},\ldots,x_{j,t-1}\}\subseteq B$, it follows that each vertex $x_{j}'$ is adjacent to at most one vertex in $B\cap \{c_{1},\ldots,c_{p}\}$. In view of this, $|B\cap \{c_{1},\ldots,c_{p}\}|\leq q$. Moreover, if $|B\cap \{c_{1},\ldots,c_{p}\}|<q$, then at least one vertex $b_{r}$ necessarily belongs to $B$ since $|B|\geq p+(3t-2)q$. This contradicts the fact that $B$ is an independent set in $G$. Thus, $|B\cap \{c_{1},\ldots,c_{p}\}|=q$. It is now easy to see that $\{C_{r}\in C:\, c_{r}\in B\}$ is a solution to the instance of X3C.

Conversely, assume that the instance of X3C has a solution $C'\subseteq C$ of cardinality $q$. It is then a routine matter to check that $\{a_{1},\ldots,a_{p}\}\bigcup(\bigcup_{j=1}^{3q}\{x_{j,1},\ldots,x_{j,t-1}\})\cup \{c_{r}:\, C_{r}\in C'\}$ is a $t$FI set in $G$ of cardinality $k$.  This completes the proof. 
\end{proof}

In contrast with the NP-completeness result in Theorem~\ref{NPC1}, the problem is efficiently solvable in trees.

\begin{theorem}
For each positive integer $t$, there exists a linear-time algorithm that computes the value $\alpha_t^{f}(T)$ for any tree $T$. 
\end{theorem}
\begin{proof}
The algorithm that provides the proof is based on the following greedy approach.\vspace{-1.5mm} 
\begin{itemize}
\item Root a tree $T$ in a non-leaf vertex $r$ of $T$, and order the vertices in bottom-to-top ordering. More precisely, we start with the vertices at distance ${\rm ecc}_G(r)$ from $r$, and then for any $k\in [{\rm ecc}_G(r)-1]\cup \{0\}$, the vertices at distance $k$ from $r$ appear after the vertices at distance $k+1$ from $r$.\vspace{-1.5mm} 
\item Initially, set $A=\emptyset$. Process the vertices of $T$ in the described order by adding a vertex to the set $A$ whenever it does not violate the condition of $A$ being a $t$-frugal independent set. 
\end{itemize}\vspace{-1.5mm}
{\bf Claim.} $A$ is an $\alpha_t^{f}(T)$-set.\vspace{1mm}\\
{\it Proof \emph{(}of Claim\emph{)}}. Let $A_i$ be the subset of $A$ of the vertices that were added in the first $i$ steps to $A$, where $i\in [|A|]$. (Thus, $A_1$ is a singleton containing the first vertex added to $A$ by the above algorithm.) The proof is by induction on $i$, where we claim that there exists an $\alpha_t^{f}(T)$-set $B_i$ that contains $A_i$ for each $i\ge 1$. The proof of the basis of induction (that $A_1$ belongs to an $\alpha_t^{f}(T)$-set) will be omitted, since it is easy and also very similar to the proof of the inductive step, which we present next.   

Given $i\in[|A|-1]$, the induction hypothesis is that there exists an $\alpha_t^{f}(T)$-set $B_i$ that contains $A_i$. Let $a_{i+1}\in A_{i+1}\setminus A_i$ be the vertex, which is added to $A$ immediately after all vertices of $A_i$ have been added. We may assume that $a_{i+1}\notin B_i$, for otherwise $B_{i+1}=B_i$ yields the desired $\alpha_t^{f}(T)$-set that contains $A_{i+1}$ and we are done. Let $p$ be the parent of $a_{i+1}$, if it exists, and let $C$ be the set of children of $p$. We may assume that $p\notin B_i$, for otherwise $(B_{i}\setminus\{p\})\cup \{a_{i+1}\}$ is a $t$FI set of $T$, which contains $a_{i+1}$ and we are done. (This is because $A$ is built with respect to bottom-to-top order, and so adding $a_{i+1}$ to $A_i$ does not violate the $t$-frugal independence condition in the subtree of $T$ with $a_{i+1}$ as its root.)
Next, assume that $|C\cap B_i|<t$. We infer that $B'=(B_{i}\setminus \{g\})\cup \{a_{i+1}\}$, where $g$ is the parent of $p$ in $T$ if it exists, is a $t$FI set of $T$ of cardinality at least $\alpha_t^{f}(T)$. Clearly, $|B'|=\alpha_t^{f}(T)$, and $B'$ contains $A_{i+1}$. Thus, we may write $B_{i+1}=B'$, and the inductive step is proved. 
Finally, assume that $|C\cap B_i|=t$. Since $A$ is a $t$FI set of $T$ and $a_{i+1}\in C\cap A$, we infer that there exists $c\in (C\setminus A)\cap B_i$. Let $B'=(B_i\setminus \{c\})\cup \{a_{i+1}\}$. Note that $B'$ is a $t$FI set of $T$ that contains $A_{i+1}$, which proves the inductive step by setting $B_{i+1}=B'$. \smallqed

\smallskip 

Noting that the algorithm is clearly linear, the proof is complete.
\end{proof}

\section{General bounds}
\label{sec:generalbounds}

Since the problems of computing $\chi_{t}^{f}$ and $\alpha_{t}^{f}$ are NP-hard, it is desirable to bound their values in terms of several invariants of the graph. Accordingly, we bound these parameters from below and above.

Let $G$ be an $r$-partite graph, where $r\ge 2$, such that for each partite set $X$ and $v\in X$, the vertex $v$ has precisely $t$ neighbors in every other partite set. Let $\Psi_{t}$ be the family of all such graphs $G$. As simple examples note that $C_{2n}\in \Psi_2$ for any $n\ge 2$, and a complete multipartite graph whose parts are of cardinality $t$ belongs to $\Psi_t$. 

\begin{theorem}\label{General1}
If $t$ is a positive integer and $G$ is a graph of size $m$, then
\begin{center}
$\chi_{t}^{f}(G)\geq \dfrac{1}{2}+\sqrt{\dfrac{1}{4}+\dfrac{2m}{t\alpha_{t}^{f}(G)}}$
\end{center}
with equality if and only if $G\in \Psi_{t}$.
\end{theorem}
\begin{proof}
The lower bound trivially holds for edgeless graphs. So, we assume that $\chi_{t}^{f}(G)\geq2$. Let $\mathbb{B}=\{B_{1},\ldots,B_{\chi_{t}^{f}(G)}\}$ be a $\chi_{t}^{f}(G)$-coloring. Without loss of generality, we may assume that $|B_{1}|\leq \ldots\leq|B_{\chi_{t}^{f}(G)}|$. Then, 
\begin{equation}\label{Ine}
\begin{array}{lcl}
m&=&\sum_{s=1}^{\chi_{t}^{f}(G)-1}\sum_{t=s+1}^{\chi_{t}^{f}(G)}|[B_{s},B_{t}]|\leq t\sum_{s=1}^{\chi_{t}^{f}(G)-1}|B_{s}|(\chi_{t}^{f}(G)-s)\vspace{1mm}\\
&\leq& t\alpha_{t}^{f}(G)\sum_{s=1}^{\chi_{t}^{f}(G)-1}(\chi_{t}^{f}(G)-s)=t\alpha_{t}^{f}(G)\big{(}\dfrac{\chi_{t}^{f}(G)(\chi_{t}^{f}(G)-1)}{2}\big{)}.
\end{array}
\end{equation}
Therefore, $t\alpha_{t}^{f}(G)\chi_{t}^{f}(G)^{2}-t\alpha_{t}^{f}(G)\chi_{t}^{f}(G)-2m\geq0$. Solving this inequality for $\chi_{t}^{f}(G)$, we obtain the desired lower bound on $\chi_{t}^{f}(G)$.

Now, let us prove that any graph $G\in \Psi_{t}$ attains the lower bound.  Let $X_{1},\ldots,X_{r}$ be the partite sets of $G$. For any two distinct indices $i,j\in[r]$, $|X_{i}|=|X_{j}|$ follows from the fact that every vertex in $X_{i}$ (resp. $X_{j}$) has precisely $t$ neighbors in $X_{j}$ (resp. $X_{i}$). 
By the structure of $G$, each partite set $X_{t}$ is a $t$FI set. This shows that $|X_{t}|\leq \alpha_{t}^{f}(G)$ and that $\chi_{t}^{f}(G)\leq r$. Let $B$ be an $\alpha_{t}^{f}(G)$-set. Suppose to the contrary that $|X_{t}|<\alpha_{t}^{f}(G)$. This in particular implies that $B\nsubseteq X_{t}$. Suppose that $X_{t}\subseteq B$. In such a situation, the strict inequality $|X_{t}|<\alpha_{t}^{f}(G)$ shows that $B$ contains a vertex $v$ from a partite set $X_{j}$ with $j\neq t$. This contradicts the fact that $B$ is an independent set as $v$ has a neighbor in $X_{t}$ by the structure of $G$. Therefore, $X_{t}\nsubseteq B$. Now set $Q=\{(b,x):\, b\in B\setminus X_{t}, x\in X_{t}\setminus B\ \mbox{and}\ bx\in E(G)\}$. Since $B$ is a $t$FI set, every vertex in $X_{t}\setminus B$ is adjacent to at most $t$ vertices in $B\setminus X_{t}$. So, $|Q|\leq t|X_{t}\setminus B|=t|X_{t}|-t|B\cap X_{t}|$. On the other hand, every vertex in $B\setminus X_{t}$ has exactly $t$ neighbors in $X_{t}\setminus B$ by the structure of $G$ and since $B$ is an independent set. This shows that $|Q|=t|B\setminus X_{t}|=t|B|-t|B\cap X_{t}|$. This, together with the last inequality, results in $|B|\leq|X_{t}|$. Therefore, $\alpha_{t}^{f}(G)=|X_{1}|=\ldots=|X_{r}|=|V(G)|/r$. With this in mind, we have
\begin{center} 
$\chi_{t}^{f}(G)\geq \dfrac{1}{2}+\sqrt{\dfrac{1}{4}+\dfrac{2m}{t\alpha_{t}^{f}(G)}}=\dfrac{1}{2}+\sqrt{\dfrac{1}{4}+r^{2}-r}=r$.
\end{center}
This leads to the desired equality.

Conversely, let $G$ be a graph that attains the lower bound. Note that $G$ is a $\chi_{t}^{f}(G)$-partite graph with partite sets $B_{1},\ldots,B_{\chi_{t}^{f}(G)}$. Because the lower bound holds with equality for $G$, it necessarily follows that (\ref{Ine}) holds with equality. By the equality in the second inequality in (\ref{Ine}), we have $|B_{1}|=\ldots=|B_{\chi_{t}^{f}(G)-1}|=\alpha_{t}^{f}(G)$. We also have $|B_{\chi_{t}^{f}(G)}|=\alpha_{t}^{f}(G)$ since $B_{\chi_{t}^{f}(G)}$ is a $t$FI set in $G$ and $|B_{\chi_{t}^{f}(G)}|\geq|B_{\chi_{t}^{f}(G)-1}|$. Moreover, since equality holds in the first inequality in (\ref{Ine}), it follows that for each partite set $B_{i}$ and vertex $v\notin B_{i}$, the vertex $v$ has precisely $t$ neighbors in $B_{i}$. We now infer that $G\in \Psi_{t}$ by taking into account the fact that $\chi_{t}^{f}(G)$ has the same role as $r$ does in the description of the members in $\Psi_{t}$. This completes the proof.
\end{proof}

Given a graph $G$, a vertex $v$ with $\deg_G(v)=1$ is a {\em pendant vertex}, and a neighbor of a pendant vertex is a {\em support vertex} in $G$. A support vertex with only one pendant neighbor is a {\em weak support vertex}, otherwise it is a {\em strong support vertex}.  Let $P(G)$ and $S(G)$ be the sets of pendant vertices and support vertices of a graph $G$. If $V(G)=P(G)\cup S(G)$, then it is readily checked that $\alpha_2^{f}(G)=s+s'$, where $s=|S(G)|$ and $s'$ is the number of strong support vertices in $G$. If $P(G)\cup S(G)\subsetneq V(G)$, we let $\delta^{*}=\delta^{*}(G)=\min\{\deg_{G}(v):\, \textrm{$v$ is neither a pendant vertex nor a support vertex}\}$. Obviously, $\delta^{*}(G)$ is well defined and $\delta^{*}(G)\geq2$.

Before proceeding further, we recall the Erd\H{o}s-Gallai degree sequence characterization. A sequence $d_{1}\geq d_{2}\geq \ldots\geq d_{n}$ consists of the vertex degrees of a simple graph if and only if $\sum_{j=1}^{n}d_{j}$ is even and $\sum_{j=1}^{k}d_{j}\leq k(k-1)+
\sum_{j=k+1}^{n}\min \{k,d_{j}\}$ for $k\in[n]$. 

\begin{theorem}\label{Pendant}
Let $G$ be a graph of order $n\geq2$ with $p$ pendant, $s$ support, and $s'$ strong support vertices. If $V(G)=P(G)\cup S(G)$, then $\alpha_2^{f}(G)=s+s'$. Otherwise,
\begin{center}
$\alpha_2^{f}(G)\leq \dfrac{2(n-p)+(s+s')(\delta^{*}+1)}{\delta^{*}+2}$,
\end{center}
and this bound is sharp.
\end{theorem}
\begin{proof}
Since the situation when $V(G)=P(G)\cup S(G)$ is straightforward, we may assume that $P(G)\cup S(G)\subsetneq V(G)$. This in particular shows that $n\geq3$. Let $B$ be an $\alpha_2^{f}(G)$-set. Let $u$ be a weak support vertex and $v$ be the unique pendant vertex adjacent to $u$. Assume that $v\notin B$. If $u\in B$, then it has no neighbor in $B$. Therefore, $(B\setminus \{u\})\cup \{v\}$ is also an $\alpha_2^{f}(G)$-set. If $u\notin B$, then it necessarily has precisely two neighbors, say $x$ and $y$, in $B$. In this case, $(B\setminus \{y\})\cup \{v\}$ is again an $\alpha_2^{f}(G)$-set. So, we may assume that all leaves adjacent to weak support vertices belong to $B$.

Now let $u$ be a strong support vertex in $G$. Note by definition that at most two leaves adjacent to $u$ belong to $B$. If $u\in B$, then no neighbor of $u$ is in $B$. In such a situation, $(B\setminus \{u\})\cup \{x,y\}$ is an I$2$F set in $G$, in which $x$ and $y$ are any two leaves adjacent to $u$. This contradicts the maximality of $B$. Therefore, $u\notin B$. Let $P_{u}$ be the set of pendant vertices adjacent to $u$ and $x,y\in P_{u}$. Note that $u$ has exactly two neighbors in $B$, for otherwise $B'=\big{(}B\setminus(N_{G}(u)\big{)}\cup \{x,y\}$ would be an I$2$F set in $G$ of cardinality greater than $|B|$, a contradiction. With this in mind, $B'$ is necessarily an $\alpha_2^{f}(G)$-set. 

Summing up, we have proved that there exists an $\alpha_2^{f}(G)$-set $B$ having\\
$(i)$ the unique pendant vertex in $P_{u}$ for each weak support vertex $u$, and\\
$(ii)$ two pendant vertices in $P_{u}$ for each strong support vertex $u$.\\
In particular, no support vertex belongs to $B$.

Taking the statements $(i)$ and $(ii)$ into account, we have $|B\cap(\cup_{u\in S(G)}P_{u})|=s+s'$, and every vertex in $B\cap(\cup_{u\in S(G)}P_{u})$ has precisely one neighbor in $V(G)\setminus B$. Moreover, every vertex in $B\setminus \big{(}P(G)\cup S(G)\big{)}$ has at least $\delta^{*}$ neighbors in $V(G)\setminus B$. Hence,
\begin{equation}\label{EQ1}
|[B,V(G)\setminus B]|\geq s+s'+(|B|-s-s')\delta^{*}.
\end{equation}

Note that no vertex in $(\cup_{u\in S(G)}P_{u})\setminus B$ has a neighbor in $B$. Moreover, each vertex in $(V(G)\setminus B)\setminus(\cup_{u\in S(G)}P_{u})$ has at most two neighbors in $B$. This leads to 
\begin{equation}\label{EQ2}
|[B,V(G)\setminus B]|\leq2s+2\Big{(}n-|B|-s-\big{(}p-(s+s')\big{)}\Big{)}.
\end{equation}

Together inequalities (\ref{EQ1}) and (\ref{EQ2}) imply that
\begin{equation*}
|B|\leq \dfrac{2(n-p)+(s+s')(\delta^{*}+1)}{\delta^{*}+2}.
\end{equation*}

That the bound is sharp may be seen as follows. Let $r\geq3$ be an integer, and let $G'$ be a bipartite graph with partite sets $X$ and $Y$ such that\\
$\bullet$ $\deg_{G'}(x)=r$ and $\deg_{G'}(y)=2$ for every $x\in X$ and $y\in Y$, and\\
$\bullet$ $r|Y|\equiv0$ (mod $2$).

By the Erd\H{o}s-Gallai degree sequence characterization, we can construct an $(r-2)$-regular graph on the vertices in $Y$. Let $H$ be the resulting $r$-regular graph. It is clear from the construction of $H$ that $X$ is an I$2$F set in $H$. Moreover, $r|X|=2|Y|=2(n(H)-|X|)$. Therefore, $\alpha_2^{f}(H)\geq|X|=2n(H)/(r+2)$. This coincides with the upper bound in the statement of the theorem by taking $P(H)=S(H)=\emptyset$ and $\delta^{*}(H)=r$ into account. Therefore, the upper bound is sharp for $H$.
\end{proof}

The upper bound in Theorem \ref{Pendant} is also sharp for graphs with minimum degree $1$, in particular for trees. The path $P_{2t+1}$, for any integer $t\geq2$, attains the bound as $\chi_{2}^{f}(P_{2t+1})=t+1$ and $\delta^{*}(P_{2t+1})=2$. Moreover, let $T$ be obtained from the star $K_{1,t}$ by subdividing each edge exactly once. It is then easily seen that the upper bound is sharp for $T$ by taking $\chi_{2}^{f}(T)=t+1$ and $\delta^{*}(T)=t$ into account.

Theorem~\ref{General1} provided a lower bound on $\fk(G)$ for any graph $G$. In the next result, we present an upper bound on $\fd(G)$, this time restricted to triangle-free graphs $G$. 

\begin{proposition}\label{TF}
For any triangle-free graph $G$ of order $n$,
\begin{center}
$\chi_{2}^{f}(G)\leq \Big{\lfloor}\dfrac{n-\alpha_2^{f}(G)+4}{2}\Big{\rfloor}$
\end{center}
and this bound is sharp.
\end{proposition}
\begin{proof}
If $n\leq \alpha_2^{f}(G)+2$, then the upper bound is easily verified. So, we may assume that $n\geq \alpha_2^{f}(G)+3$. Let $B$ be an $\alpha_2^{f}(G)$-set and set $G'=G[V(G)\setminus B]$. Clearly, $G'$ is a triangle-free graph as well. With this in mind and since $|V(G')|\geq3$, once can partition $V(G')$ into $t$ subsets $B_{1},\ldots,B_{t}$ in such a way that
\begin{itemize}
    \item  $B_{i}$ consists of two nonadjacent vertices for each $i\in[t-1]$, and
    \item  $|B_{t}|\in \{1,2\}$.
\end{itemize}
If $|B_{t}|=1$ or $B_{t}$ consists of two nonadjacent vertices, then $\{B,B_{1},\ldots,B_{t}\}$ is a $2$-frugal coloring of $G$. Therefore, $\chi_{2}^{f}(G)\leq2+(n-\alpha_2^{f}(G)-1)/2=(n-\alpha_2^{f}(G)+3)/2$ or $\chi_{2}^{f}(G)\leq1+\big{(}n-\alpha_2^{f}(G)\big{)}/2=(n-\alpha_2^{f}(G)+2)/2$ if $|B_{t}|=1$ or $|B_{t}|=2$, respectively. 

Now, assume that $B_{t}=\{x,y\}$ such that $xy\in E(G)$. In this case, $\mathcal{B}=\big{\{}B,B_{1},\ldots,B_{t-1},\{x\},\{y\}\big{\}}$ is a $2$-frugal coloring of $G$. Thus, $\chi_{2}^{f}(G)\leq3+(n-\alpha_2^{f}(G)-2)/2=(n-\alpha_2^{f}(G)+4)/2$.

In either case, the resulting inequality leads to the desired upper bound. The bound is sharp for the complete bipartite graph $K_{2s+1,2t+1}$, for all positive integers $s$ and $t$, by taking into account the fact that $\big{(}\chi_{2}^{f}(K_{2s+1,2t+1}),|V(K_{2s+1,2t+1})|,\alpha_2^{f}(K_{2s+1,2t+1})\big{)}=(s+t+2,2s+2t+2,2)$.
\end{proof}

Note that the assumption ``being triangle-free'' in Proposition \ref{TF} cannot be removed. For instance, $\chi_{2}^{f}(K_{n})=n>\lfloor(n-\alpha_2^{f}(K_{n})+4)/2\rfloor=\lfloor(n+3)/2\rfloor$ for each $n\geq4$. It should also be noted that the counterpart of Proposition \ref{TF} for the $2$-distance chromatic number does not hold. To see this, consider the star $K_{1,n-1}$ for any integer $n\geq4$. We observe that $\chi_{2}(K_{1,n-1})=n>\lfloor(n-\rho_{2}(K_{1,n-1})+4)/2\rfloor=\lfloor(n+3)/2\rfloor$.

\section{Subcubic graphs}
\label{sec:subcubic}

Unlike graphs with maximum degree $2$ whose frugal colorings are trivial, subcubic graphs (that is, graphs with maximum degree $3$) already bring challenging questions with respect to frugal colorings and frugal independence. Note that for any subcubic graph $G$, $\fk(G)=\chi(G)$ and $\alpha_t^{f}(G)=\alpha(G)$ as soon as $t\ge 3$, hence we will only be interested in $2$-frugal chromatic and $2$-frugal independence numbers of subcubic graphs.

We start with an auxiliary result needed in the main theorem of this section.

\begin{lemma}\label{L3}
If $G$ is a connected graph with $\Delta(G)=3$, which contains a vertex of degree at most $2$, then $\fd(G)\le 5$. 
\end{lemma}
\begin{proof}
Let $G$ be a connected subcubic graph and let $v\in V(G)$ be a vertex with $\deg_G(v)\le 2$. 
We will present a procedure by which all vertices of $G$ will be colored by one of the colors in $[5]$ resulting in a $2$-frugal coloring of $G$.  

Consider a spanning tree $T$ of $G$ and root it at $v$. Proceed with the coloring of the vertices of $G$ in a bottom-to-top order, starting with the leaves of $T$ and assigning a color to a vertex only when all of its children have been colored. Clearly, when a vertex $u$, where $u\ne v$, is being colored, there exists exactly one neighbor of $u$, namely its parent, which has not yet been colored. Hence, we claim that it is possible to color $u$ with one of the five colors while maintaining the property that the partially colored graph is assigned a $2$-frugal coloring. Indeed, there are at most four colors that are forbidden for $u$: the colors given to its (at most two) children $u_1$ and $u_2$, and possibly the colors of the children of $u_i$ if the same color appears on both children of $u_i$, for $i\in [2]$ (again there are at most two such colors). Hence, every vertex $u\in V(G)\setminus\{v\}$ can be colored in the desired way. Finally, since $v$ has at most two children, we can use the same argument as before to color $v$, which results in a $2$-frugal coloring of $G$ using at most $5$ colors. 
\end{proof}

As an immediate consequence of Observation \ref{ob:boundDelta}, we have $\fd(G)\leq7$ for all subcubic graphs. In the next result, we prove that this upper bound can be improved.

\begin{theorem}
\label{boundsDelta=3}
If $G$ is a graph with $\Delta(G)=3$, then $3\le \fd(G)\le 5$.    
\end{theorem}
\begin{proof}
Clearly, if $u$ is a vertex of maximum degree $3$ in $G$, then any $2$-frugal coloring assigns at least three colors to the vertices in $N_{G}[u]$. This proves the lower bound.
 
For the proof of the upper bound, it suffices to restrict to cubic graphs $G$ due to Lemma \ref{L3}. We proceed by contradiction and suppose that there exists a cubic graph $G$ such that $\chi_2^{f}(G)\geq6$. Let $uv\in E(G)$ and $G'=G-uv$. By Lemma~\ref{L3}, $G'$ admits a $2$-frugal coloring with $5$ colors. Since $G$ is cubic, we have $\deg_{G'}(u)=\deg_{G'}(v)=2$. For each $w\in\{u,v\}$, let $N_{G'}(w)=\{w_1,w_2\}$. For each $j\in[2]$, let $N_{G'}(w_j)=\{w_{j1},w_{j2},w\}$, where $w_{j1}=w_i$ for some $i\neq j$ if $w_j$ is adjacent to $w_i$. Moreover, we set $N_{G'}(w_{ij})=\{w_{ij1},w_{ij2},w_{i}\}$ for each $i,j\in[2]$. (Note that some of the vertices defined above may coincide. However, this does not affect the arguments in the proof.)
	
Let $\varphi:V(G')\to[5]$ be a $\chi_2^{f}(G')$-coloring. If $\varphi(u)\neq\varphi(v)$ and $|\{\varphi(x_i),\varphi(y)\mid i\in[2]\}|\geq2$ for distinct $x,y\in\{u,v\}$, then $\varphi$ is also a $2$-frugal coloring of $G$, which implies that $\chi_2^{f}(G)\le 5$, a contradiction. In view of this fact, we distinguish two cases.\vspace{1mm}\\
\textit{Case 1.} $\varphi(u)=\varphi(v)=a$ for some $a\in[5]$. In such a situation, we need to distinguish three more possibilities.\vspace{0.75mm}\\
\textit{Subcase 1.1.} $\varphi(u_{11})=\varphi(u_{21})=a$. Since $|[5]\setminus\{a,\varphi(u_1),\varphi(u_2)\}|\ge 2$, it follows that there exists a vertex $b\in[5]\setminus\{a,\varphi(u_1),\varphi(u_2)\}$ such that $|\{b,\varphi(v_1),\varphi(v_2)\}|\ge 2$. We now define $\rho$ by $\rho(u)=b$ and $\rho(x)=\varphi(x)$ for all $x\in V(G)\setminus\{u\}$. Then, $\rho$ is a $2$-frugal coloring of $G$ with five colors, which is a contradiction.\vspace{0.75mm}\\
\textit{Subcase 1.2.} $\varphi(u_{11})=a$ and $a\notin\{\varphi(u_{21}),\varphi(u_{22})\}$. We need to consider two possibilities depending on the behavior of $\varphi(u_{21})$ and $\varphi(u_{22})$.\vspace{0.75mm}\\
\textit{Subcase 1.2.1.} $\varphi(u_{21})=\varphi(u_{22})$. Then, $Q=[5]\setminus\{a,\varphi(u_1),\varphi(u_2),\varphi(u_{21})\}$ has at least one color. If there exists a color $b\in Q$ such that $|\{b,\varphi(v_1),\varphi(v_2)\}|\geq2$, then $\rho(u)=b$ and $\rho(x)=\varphi(x)$ for all $x\in V(G)\setminus\{u\}$ defines a $2$-frugal coloring of $G$ with five colors, a contradiction. Therefore, we may assume $Q=\{b\}$ and $\varphi(v_{1})=\varphi(v_{2})=b$. Note that
\begin{itemize}\vspace{-1.75mm}
\item if $\{\varphi(u_{2i1}),\varphi(u_{2i2})\}\neq \{a\}$ for each $i\in[2]$, then $\rho(u)=\varphi(u_2)$, $\rho(u_{2})=a$ and $\rho(x)=\varphi(x)$ for all $x\in V(G)\setminus\{u,u_2\}$ defines a $2$-frugal coloring of $G$, and\vspace{-1.75mm}
\item if $\{\varphi(u_{2i1}),\varphi(u_{2i2})\}\neq \{b\}$ for each $i\in[2]$, then $\rho(u)=\varphi(u_2)$, $\rho(u_2)=b$ and $\rho(x)=\varphi(x)$ for all $x\in V(G)\setminus\{u,u_2\}$ gives us a $2$-frugal coloring of $G$.
\end{itemize}\vspace{-1.75mm}
In either case, $\rho$ uses five colors. This is a contradiction. Thus, without loss of generality, we may assume that $\varphi(u_{211})=\varphi(u_{212})=a$ and $\varphi(u_{221})=\varphi(u_{222})=b$. In such a situation, the assignment $\rho(u_2)=\varphi(u_1)$, $\rho(u)=\varphi(u_2)$ and $\rho(x)=\varphi(x)$ for all $x\in V(G)\setminus\{u,u_2\}$ gives a $2$-frugal coloring of $G$ using five colors, a contradiction.\vspace{0.75mm}\\
\textit{Subcase 1.2.2.} $\varphi(u_{21})\neq \varphi(u_{22})$. It is then easy to see that there is a color $b\in[5]\setminus \{a,\varphi(u_{1}),\varphi(u_{2})\}$ such that $|\{b,\varphi(v_{1}),\varphi(v_{2})\}|\geq2$. Then, changing the color of $u$ from $a$ to $b$ and keeping the other colors fixed defines a $2$-frugal coloring of $G$ with five colors, which is a contradiction.\vspace{0.75mm}\\
\textit{Subcase 1.3.} $\varphi(u)=a\notin\{\varphi(u_{11}),\varphi(u_{12}),\varphi(u_{21}),\varphi(u_{22})\}$. First, we suppose that $\varphi(u_{ij})=\varphi(u_{it})$ for all $i,j,t\in[2]$ with $j\neq t$. If $\varphi(u_{1i})=\varphi(u_{2i})$ for $i\in[2]$, then $Q=[5]\setminus \{\varphi(u),\varphi(u_{i}),\varphi(u_{i1}),\varphi(u_{i2})\mid i\in[2]\}$ is nonempty. If  there exists a vertex $b\in Q$ such that $|\{b,\varphi(v_1),\varphi(v_2)\}|\geq2$. In such a situation, the assignment $\rho(u)=b$ and $\rho(x)=\varphi(x)$ for all $x\in V(G)\setminus\{u\}$ gives a $2$-frugal coloring of $G$ using five colors, a contradiction. Therefore, $\varphi(v_1)=\varphi(v_2)=b$, where $b$ is the unique member of $Q$. We then observe that
\begin{itemize}\vspace{-1.75mm}
\item if $\{\varphi(u_{1i1}),\varphi(u_{1i2})\}\neq \{a\}$ for each $i\in[2]$, then $\rho(u)=\varphi(u_1)$, $\rho(u_1)=a$ and $\rho(x)=\varphi(x)$ for all $x\in V(G)\setminus\{u,u_1\}$ gives us a $2$-frugal coloring of $G$, and\vspace{-1.75mm}
\item if $\{\varphi(u_{1i1}),\varphi(u_{1i2})\}\neq \{b\}$ for each $i\in[2]$, then $\rho(u)=\varphi(u_1)$, $\rho(u_1)=b$ and $\rho(x)=\varphi(x)$ for any other vertex $x$ defines a $2$-frugal coloring of $G$. 
\end{itemize}\vspace{-1.75mm}
In either case, $\rho$ uses five colors, which is impossible. Therefore, we may assume that $\varphi(u_{11i})=a$ and $\varphi(u_{12i})=b$ for all $i\in[2]$. With this in mind, we define $\rho$ by $\rho(u_1)=\varphi(u_2)$, $\rho(u)=\varphi(u_1)$ and $\rho(x)=\varphi(x)$ for all $x\in V(G)\setminus\{u,u_1\}$. It is then easily checked that $\rho$ is a $2$-frugal coloring of $G$ with five colors, a contradiction.

  Suppose now that $\varphi(u_{1i})\neq\varphi(u_{2i})$ for $i\in[2]$. Let $R=\{a,\varphi(u_i),\varphi(u_{ij})|i,j\in[2]\}$. Clearly, $|R|\geq3$. If $|R|=3$, there exists a color $b\in[5]\setminus R$ such that $|\{b,\varphi(v_{1}),\varphi(v_{2})\}|\geq2$. In such a case, changing  the color of $u$ from $a$ to $b$ and keeping the other colors fixed gives a $2$-frugal coloring of $G$ with five colors, a contradiction. Therefore, $|R|\geq4$. Suppose first that $|R|=4$ and that $b$ is the unique member of $[5]\setminus R$. If 
  $\{\varphi(v_{1}),\varphi(v_{2})\}\ne \{b\}$, then $\rho(u)=b$ and $\rho(x)=\varphi(x)$ for any other vertex $x$ defines a $2$-frugal coloring of $G$ with five colors, which is impossible. Therefore, we may suppose that $\varphi(v_{1})=\varphi(v_{2})=b$.
    Note that at least one of the vertices in $\{u_1, u_2\}$, say, by symmetry $u_1$, receives a color different from $a,\varphi(u_{11})$ and $\varphi(u_{21})$.
We need to differentiate three more possibilities.\vspace{0.75mm}\\
\textit{Subcase 1.3.1.} $\{\varphi(u_{1i1}),\varphi(u_{1i2})\}\neq \{\varphi(u_{2i})\}$ and $\{\varphi(u_{2i1}),\varphi(u_{2i2})\}\neq \{\varphi(u_{1i})\}$ for each $i\in[2]$. Then, the assignment $\big(\rho(u_{1}),\rho(u_{2}),\rho(u)\big)=\big(\varphi(u_{21}),\varphi(u_{11}),\varphi(u_{1})\big)$ and $\rho(x)=\varphi(x)$ for each $x\in V(G)\setminus \{u,u_{1},u_{2}\}$ gives us a $2$-frugal coloring of $G$.\vspace{0.75mm}\\
\textit{Subcase 1.3.2.} $\{\varphi(u_{1i1}),\varphi(u_{1i2})\}=\{\varphi(u_{2i})\}$ for some $i\in[2]$ and $\{\varphi(u_{2i1}),\varphi(u_{2i2})\}\neq \{\varphi(u_{1i})\}$ for each $i\in[2]$. We may assume that $\{\varphi(u_{111}),\varphi(u_{112})\}=\{\varphi(u_{21})\}$. If $\{\varphi(u_{121}),\varphi(u_{122})\}\neq \{b\}$, then $\big(\rho(u_{1}),\rho(u_{2}),\rho(u)\big)=\big(b,\varphi(u_{11}),\varphi(u_{1})\big)$ and $\rho(x)=\varphi(x)$ for any other vertex $x$ defines a $2$-frugal coloring of $G$. Otherwise, $\big(\rho(u_{1}),\rho(u_{2}),\rho(u)\big)=\big(a,\varphi(u_{11}),\varphi(u_{1})\big)$ and $\rho(x)=\varphi(x)$ for any other vertex $x$ defines a $2$-frugal coloring of $G$.\vspace{0.75mm}\\ 
\textit{Subcase 1.3.3.} We may assume, without loss of generality, that $\{\varphi(u_{111}),\varphi(u_{112})\}=\{\varphi(u_{21})\}$ and $\{\varphi(u_{211}),\varphi(u_{212})\}=\{\varphi(u_{11})\}$. We need to consider the following possibilities.
\begin{itemize}\vspace{-1.75mm}
\item $\{\varphi(u_{121}),\varphi(u_{122})\}=\{b\}$ and $\{\varphi(u_{221}),\varphi(u_{222})\}\neq \{b\}$. Then, $\big(\rho(u_{1}),\rho(u_{2}),\rho(u)\big)=\big(a,b,\varphi(u_{1})\big)$ and $\rho(x)=\varphi(x)$ for every $x\in V(G)\setminus \{u,u_{1},u_{2}\}$ defines a $2$-frugal coloring of $G$.\vspace{-1.75mm}
\item $\{\varphi(u_{121}),\varphi(u_{122})\}\neq \{b\}$ and $\{\varphi(u_{221}),\varphi(u_{222})\}\neq \{b\}$. Then, the assignment $\big(\rho(u_{1}),\rho(u_{2}),\rho(u)\big)=\big(b,b,\varphi(u_{1})\big)$ and $\rho(x)=\varphi(x)$ for each $x\in V(G)\setminus \{u,u_{1},u_{2}\}$ gives a $2$-frugal coloring of $G$.\vspace{-1.75mm}
\item $\varphi(u_{121})=\varphi(u_{122})=\varphi(u_{221})=\varphi(u_{222})=b$.
In such a situation, the following possibilities arise.\vspace{0.25mm}\\
$(i)$ $\{\varphi(u_{11i1}),\varphi(u_{11i2})\}\neq \{b\}$ for every $i\in[2]$. Then, $\rho(u_{11})=b$, $\rho(u)=\varphi(u_{11})$ and $\rho(x)=\varphi(x)$ for all $x\in V(G)\setminus \{u,u_{11}\}$ is a $2$-frugal coloring of $G$. \vspace{0.25mm}\\
$(ii)$ $\{\varphi(u_{11i1}),\varphi(u_{11i2})\}\neq \{\varphi(u_{1})\}$ for every $i\in[2]$. In such a case, $\big(\rho(u_{11}),\rho(u_{1}),\rho(u)\big)=\big(\varphi(u_{1}),a,\varphi(u_{11})\big)$ and $\rho(x)=\varphi(x)$ for any other vertex $x$ defines a $2$-frugal coloring of $G$.\vspace{0.25mm}\\
$(iii)$ We may suppose that $\varphi(u_{1111})=\varphi(u_{1112})=b$ and $\varphi(u_{1121})=\varphi(u_{1122})=\varphi(u_{1})$. Then, $\rho(u_{11})=a$, $\rho(u)=\varphi(u_{11})$ and $\rho(x)=\varphi(x)$ for all $x\in V(G)\setminus \{u,u_{11}\}$ gives us a $2$-frugal coloring of $G$.\end{itemize}\vspace{-1.75mm}

In view of the discussion above, we assume that $|R
|=5$. This in particular implies that $\varphi(u_1)\neq\varphi(u_2)$. Thus we may assume without loss of generality, changing the roles of $u_1$ and $u_2$ if necessary, that $|\{\varphi(u_{1}),\varphi(v_{1}),\varphi(v_{2})\}|\geq2$. We observe that 
\begin{itemize}\vspace{-1.75mm}
\item if $\{\varphi(u_{1i1}),\varphi(u_{1i2})\}\neq \{\varphi(u_{2})\}$ for each $i\in[2]$, then $\rho(u)=\varphi(u_1)$, $\rho(u_1)=\varphi(u_2)$ and $\rho(x)=\varphi(x)$ for any other vertex $x$ is a $2$-frugal coloring of $G$, and\vspace{-1.75mm}
\item if $\{\varphi(u_{1i1}),\varphi(u_{1i2})\}\neq \{\varphi(u_{21})\}$ for each $i\in[2]$, then $\rho(u)=\varphi(u_1)$, $\rho(u_1)=\varphi(u_{21})$ and $\rho(x)=\varphi(x)$ for all $x\in V(G)\setminus\{u,u_1\}$ gives us a $2$-frugal coloring of $G$.
\end{itemize}\vspace{-1.75mm}
In either case, $\rho$ assigns five colors to the vertices of $G$, a contradiction. Hence, we may suppose that $\varphi(u_{111})=\varphi(u_{112})=\varphi(u_{2})$ and $\varphi(u_{121})=\varphi(u_{122})=\varphi(u_{21})$. We now define $\rho$ by $\rho(u_1)=a$, $\rho(u)=\varphi(u_1)$ and $\rho(x)=\varphi(x)$ for all $x\in V(G)\setminus\{u,u_1\}$. It is readily checked that $\rho$ is a $2$-frugal coloring of $G$ using five colors, which is again a contradiction.
	
Next, we suppose that $\varphi(u_{11})=\varphi(u_{12})$ and $\varphi(u_{21})\neq\varphi(u_{22})$.  We set $R=[5]\setminus\{a,\varphi(u_1),\varphi(u_2),\varphi(u_{11})\}$. If there is a color $h\in R$ such that $|\{h,\varphi(v_{1}),\varphi(v_{2})\}|\geq2$, then we derive a contradiction as above. Otherwise, we can write $R=\{h\}$ and $\varphi(v_{1})=\varphi(v_{2})=h$. We observe that
\begin{itemize}\vspace{-1.75mm}
\item if $\{\varphi(u_{1i1}),\varphi(u_{1i2})\}\neq \{\varphi(u_{2})\}$ for each $i\in[2]$, then $\rho(u_1)=\varphi(u_2)$, $\rho(u)=\varphi(u_1)$ and $\rho(x)=\varphi(x)$ for any other vertex $x$ is a $2$-frugal coloring of $G$, and\vspace{-1.75mm}
\item if $\{\varphi(u_{1i1}),\varphi(u_{1i2})\}\neq \{h\}$ for each $i\in[2]$, then $\rho(u_1)=h$, $\rho(u)=\varphi(u_1)$ and $\rho(x)=\varphi(x)$ for all $x\in V(G)\setminus \{u,u_{1}\}$ is a $2$-frugal coloring of $G$.
\end{itemize}\vspace{-1.75mm}
In either case, $\rho$ uses five colors, a contradiction. Due to this, we may assume that $\varphi(u_{111})=\varphi(u_{112})=\varphi(u_{2})$ and $\varphi(u_{121})=\varphi(u_{122})=h$. In such a situation, $\rho(u_1)=a$, $\rho(u)=\varphi(u_1)$ and $\rho(x)=\varphi(x)$ for all $x\in V(G)\setminus \{u,u_{1}\}$ is a $2$-frugal coloring of $G$ with five colors, which is again a contradiction.
	
Finally, suppose that $\varphi(u_{11})\neq\varphi(u_{12})$ and $\varphi(u_{21})\neq\varphi(u_{22})$. Choose $h\in[5]\setminus\{\varphi(u),\varphi(u_1),\varphi(u_2)\}$ such that $|\{h,\varphi(v_1),\varphi(v_2)\}|\ge 2$. Then, changing the color of $u$ from $a$ to $h$ and keeping the other colors fixed leads to a $2$-frugal coloring of $G$ using five colors, a contradiction.\vspace{1mm}\\
\textit{Case 2.} $\varphi(u)=\varphi(v_1)=\varphi(v_2)=a$ for some $a\in[5]$. We need to analyze three possibilities.\vspace{0.75mm}\\
\textit{Subcase 2.1.} $\varphi(v)=\varphi(u_1)=\varphi(u_2)=b$ for some $b\in[5]\setminus \{a\}$. Suppose first that $\varphi(w_{11})=\varphi(w_{12})$ and $\varphi(w_{21})=\varphi(w_{22})$ for all $w\in\{u,v\}$. 
Choose a color $\rho(w)\in[5]\setminus\{\varphi(w),\varphi(w_1),\varphi(w_{11}),\varphi(w_{21})\}$ for each $w\in\{u,v\}$ and set $\rho(x)=\varphi(x)$ for all $x\in V(G)\setminus\{u,v\}$. Clearly, $\rho$ is a $2$-frugal coloring of $G'$ with five colors. If $\rho(u)\neq \rho(v)$, it is then easy to check that $\rho$ is a $2$-frugal coloring of $G$ as well, a contradiction. However,  $\rho(u)=\rho(v)$ leads to a contradiction as proved in Case 1. 

We may next suppose that there exists $w\in\{u,v\}$ such that $\varphi(w_{11})\neq\varphi(w_{12})$, say $w=u$. We consider two possibilities depending on $\varphi(u_{21})$ and $\varphi(u_{22})$.\vspace{0.75mm}\\
\textit{Subcase 2.1.1.} $\varphi(u_{21})=\varphi(u_{22})$. Let $[5]\setminus \{a,b,\varphi(u_{21})\}=\{h,h'\}$. If $\{\varphi(v_{i1}),\varphi(v_{i2})\}\neq \{h''\}$ for some $h''\in \{h,h'\}$ and each $i\in[2]$, then $\rho(u)=\rho(v)=h''$ and $\rho(x)=\varphi(x)$ for all $x\in V(G)\setminus \{u,v\}$ is a $2$-frugal coloring of $G'$ with five colors. This, in view of Case 1, results in a contradiction. Hence, we may assume that $\varphi(v_{11})=\varphi(v_{12})=h$ and $\varphi(v_{21})=\varphi(v_{22})=h'$. Let $\{c\}=[5]\setminus \{a,b,h,h'\}$. Then, the assignment $\rho(u)=h$, $\rho(v)=c$ and $\rho(x)=\varphi(x)$ for any other vertex $x$ defines a $2$-frugal coloring of $G$ using five colors, which is impossible.\vspace{0.75mm}\\
\textit{Subcase 2.1.2.} $\varphi(u_{21})\neq \varphi(u_{22})$. Because $|[5]\setminus \{a,b\}|\geq3$, it follows that there is a color $h\in [5]\setminus \{a,b\}$ such that $\{\varphi(v_{i1}),\varphi(v_{i2})\}\neq \{h\}$ for each $i\in[2]$. It is then easy to see that $\rho(u)=\rho(v)=h$ and $\rho(x)=\varphi(x)$ for any other vertex $x$ is a $2$-frugal coloring of $G'$ with five colors, which leads to a contradiction by Case 1.\vspace{0.75mm}\\
\textit{Subcase 2.2.} $\varphi(v)=\varphi(u_1)=b$ and $\varphi(v)\neq\varphi(u_2)=c$. Suppose first that $\varphi(u_{i1})\neq \varphi(u_{i2})$ for some $i\in[2]$, say $\varphi(u_{21})\neq \varphi(u_{22})$ and let $h\in[5]\setminus \{a,b,c,\varphi(u_{11})\}$. If $\{\varphi(v_{i1}),\varphi(v_{i2})\}\neq \{h\}$ for each $i\in[2]$, then the assignment $\rho(u)=\rho(v)=h$ and $\rho(x)=\varphi(x)$ for each $x\in V(G)\setminus \{u,v\}$ is a $2$-frugal coloring of $G'$ with five colors, which is a contradiction in view of Case 1. Therefore, we may assume that $\varphi(v_{11})=\varphi(v_{12})=h$. Since $|[5]\setminus \{a,b,h\}|\geq2$, there exists a color $h'\in[5]\setminus \{a,b,h\}$ such that $\{\varphi(v_{21}),\varphi(v_{22})\}\neq \{h'\}$. In such a situation, $\rho(u)=h$, $\rho(v)=h'$ and $\rho(x)=\varphi(x)$ for every $x\in V(G)\setminus \{u,v\}$ gives us a $2$-frugal coloring of $G$ using five colors, a contradiction. Therefore, $\varphi(u_{i1})=\varphi(u_{i2})$ for each $i\in[2]$. If $\varphi(u_{21})\neq h$, then assigning the color $h$ to $u$ and keeping the other colors fixed results in a $2$-frugal coloring of $G$ with five colors, which is impossible. So, $\varphi(u_{21})=\varphi(u_{22})=h$. Note that if there exists a color $h'\in [5]\setminus \{a,b,c,\varphi(u_{11})\}$, where $h'\ne h$, then $\rho(u)=h'$ and $\rho(x)=\varphi(x)$ for each $x\in V(G)\setminus \{u\}$ is a $2$-frugal coloring of $G$ using five colors, which is impossible. Therefore, $[5]\setminus \{a,b,c,\varphi(u_{11})\}=\{h\}$. We now differentiate the following cases.
\begin{itemize}\vspace{-1.75mm}
\item If $\{\varphi(u_{1i1}),\varphi(u_{1i2})\}\neq \{c\}$ for each $i\in[2]$, then $\rho(u_{1})=c$, $\rho(u)=b$ and $\rho(x)=\varphi(x)$ for each $x\in V(G)\setminus \{u,u_{1}\}$ is a $2$-frugal coloring of $G'$ using five colors.\vspace{-1.75mm}
\item If $\{\varphi(u_{1i1}),\varphi(u_{1i2})\}\neq \{h\}$ for each $i\in[2]$, then $\rho(u_{1})=h$, $\rho(u)=b$ and $\rho(x)=\varphi(x)$ for each $x\in V(G)\setminus \{u,u_{1}\}$ is a $2$-frugal coloring of $G'$ using five colors.
\end{itemize}\vspace{-1.75mm}
Taking Case 1 into account, either of the above cases leads to a contradiction. Hence, we may suppose without loss of generality that $\varphi(u_{111})=\varphi(u_{112})=c$ and $\varphi(u_{121})=\varphi(u_{122})=h$. In such a situation, $\rho(u)=b$, $\rho(u_{1})=a$ and $\rho(x)=\varphi(x)$ for each $x\in V(G)\setminus \{u,u_{1}\}$ gives us a $2$-frugal coloring of $G'$ using five colors such that $\rho(u)=\varphi(v)$. This is a contradiction due to Case 1.\vspace{0.75mm}\\ 
\textit{Subcase 2.3.} $b=\varphi(v)\neq\varphi(u_1)=c$ and $b=\varphi(v)\neq\varphi(u_2)=d$. If $\{\varphi(u_{i1}),\varphi(u_{i2})\}\neq \{b\}$ for each $i\in[2]$, then $\rho(u)=b$ and $\rho(x)=\varphi(x)$ for any other vertex $x$ is a $2$-frugal coloring of $G'$ with five colors. This is impossible in view of Case 1. So, we may assume that $\varphi(u_{11})=\varphi(u_{12})=b$ and set $Q=[5]\setminus \{a,b,c,d\}$. If $|Q|\geq2$, then there is a color $h\in Q$ such that $\{\varphi(u_{21}),\varphi(u_{22})\}\neq \{h\}$. In such a case, reassigning the color $h$ to $u$ and keeping the other colors fixed gives a $2$-frugal coloring of $G$ with five colors, a contradiction. Hence, we can write $Q=\{f\}$. If  $\{\varphi(u_{21}),\varphi(u_{22})\}\neq \{f\}$. In such a case, reassigning the color $f$ to $u$ and keeping the other colors fixed gives a $2$-frugal coloring of $G$ with five colors, a contradiction. Hence,  $\varphi(u_{21})=\varphi(u_{22})=f$. Finally, we need to distinguish the following possibilities.
\begin{itemize}\vspace{-1.75mm}
\item $\{\varphi(u_{1i1}),\varphi(u_{1i2})\}\neq \{a\}$ for each $i\in[2]$, then $\rho(u)=c$, $\rho(u_{1})=a$ and $\rho(x)=\varphi(x)$ for any other vertex $x$ is a $2$-frugal coloring of $G$.\vspace{-1.75mm}
\item $\{\varphi(u_{1i1}),\varphi(u_{1i2})\}\neq \{f\}$ for each $i\in[2]$, then $\rho(u)=c$, $\rho(u_{1})=f$ and $\rho(x)=\varphi(x)$ for each $x\in V(G)\setminus \{u,u_{1}\}$ is a $2$-frugal coloring of $G$.
\end{itemize}\vspace{-1.75mm}
In either case above, $\rho$ uses five colors, a contradiction. Therefore, we may assume that $\varphi(u_{111})=\varphi(u_{112})=a$ and $\varphi(u_{121})=\varphi(u_{122})=f$. In such a situation,  $\rho(u)=c$, $\rho(u_{1})=d$ and $\rho(x)=\varphi(x)$ for any other vertex $x$ defines a $2$-frugal coloring of $G$ with five colors, which is impossible. This completes the proof.
\end{proof}

We wonder if the upper bound in Theorem~\ref{boundsDelta=3} can be improved, and pose it as an open problem. Clearly, there are cubic graphs $G$ with $\fd(G)=4$, which are for instance the graphs $G$ of order $8$ with $\alpha_2^{f}(G)=n/4$ (see the proof of Proposition~\ref{Cubic}). Therefore, the sharp upper bound on $\fd$ for cubic graphs lies between $4$ and $5$. 

In the next result we prove that the bound from Theorem~\ref{boundsDelta=3} can be improved if a subcubic graph is claw-free. For this purpose, recall Brooks' theorem stating that $\chi(G)\le \Delta(G)$  holds for any connected graph, which is not a complete graph nor an odd cycle; see~\cite{West}.
\begin{proposition}
\label{prp:clawfree}
If $G$ is a connected claw-free cubic graph not isomorphic to $K_4$, then $\fd(G)=3$.
\end{proposition}
\begin{proof}
By Brooks' theorem, $\chi(G)\le 3$ holds for every connected cubic graph $G$ different from $K_4$. Consider a proper $3$-coloring $f$ of a claw-free cubic graph $G$, different from $K_4$, and let $u\in V(G)$. Since $G$ is claw-free, there are two vertices $v,w$ in $N_G(u)$ that are adjacent. Therefore, $f(v)\ne f(w)$, which readily implies that $f$ is a $2$-frugal coloring of $G$. Thus, $\fd(G)\le 3$. Since $G$ contains a triangle, $\fd(G)\ge 3$.  
\end{proof}

 As an immediate consequence of Theorem \ref{Pendant}, we have the upper bound $2n/5$ for the $2$FI number of any cubic graph of order $n$. Moreover, the bound is achieved by the $r$-regular graphs constructed in the proof of Theorem \ref{Pendant} with $r=3$.

\begin{proposition}\label{Cubic}
For any cubic graph $G$ of order $n$,
\begin{center}
$\dfrac{n}{4}\leq \alpha_2^{f}(G)\leq \dfrac{2n}{5}$.
\end{center}
These bounds are sharp.
\end{proposition}
\begin{proof}
Let $B$ be an $\alpha_2^{f}(G)$-set. We set $B_{0}=\{v\in V(G)\setminus B:\, |N_{G}(v)\cap B|=0\}$, $B_{1}=\{v\in V(G)\setminus B:\, |N_{G}(v)\cap B|=1\}$ and $B_{2}=\{v\in V(G)\setminus B:\, |N_{G}(v)\cap B|=2\}$. It is clear that $B_{0}$, $B_{1}$ and $B_{2}$ are pairwise disjoint and that $V(G)\setminus B=B_{0}\cup B_{1}\cup B_{2}$.

By the definition of $B_{0}$ and since $B$ is an $\alpha_2^{f}(G)$-set, it follows that every vertex in $B_{0}$ has at least one neighbor in $B_{2}$. On the other hand, every vertex in $B_{2}$ has at most one neighbor in $B_{0}$. Therefore, $|B_{0}|\leq|[B_{0},B_{2}]|\leq|B_{2}|$. In view of this, and since $|V(G)\setminus B|=|B_{0}|+|B_{1}|+|B_{2}|$, we infer that 
\begin{center}
$n-|B|=|V(G)\setminus B|\leq|B_{1}|+2|B_{2}|=|[B,V(G)\setminus B]|=3|B|$.
\end{center} 
This results in the desired lower bound. The bound is sharp for some cubic graphs such as $K_{4}$, the hypercube $Q_{3}$, the twisted cube and the graph obtained from the cycle $C_{8}$ by adding the chords between four antipodal vertices.
\end{proof}



\section{Nordhaus-Gaddum type inequalities for $\chi_{2}^{f}$}
\label{sec:NG}

In 1956, Nordhaus and Gaddum presented lower and upper bounds on the sum
and product of the chromatic numbers of a graph and its complement in terms of
the order~\cite{NoGa}. From then on, inequalities bounding $\eta(G)+\eta(\overline{G})$ and $\eta(G)\eta(\overline{G})$ are called \textit{Nordhaus-Gaddum inequalities}, where $\eta$ is any graph parameter. For comprehensive information about this subject up to 2013, the reader can consult~\cite{NGAH}.

In the statement of the following theorem, three graphs $G_1$, $G_2$ and $G_3$ appear, each of which is $4$-regular with $9$ vertices; see Figure~\ref{fig:3graphs}.

\begin{theorem}\label{N-G}
For any graph $G\notin \{G_{j},\overline{G_{j}}\}_{j=1}^{3}$ of order $n\geq2$, 
\begin{center}
$\chi_{2}^{f}(G)+\chi_{2}^{f}(\overline{G})\geq \dfrac{n}{2}+2$.
\end{center} 
Moreover, the bound is sharp.
\end{theorem}
\begin{proof}
Let $G$ be a graph, and $\mathcal{J}$ be the color classes of a $\chi_{2}^{f}(G)$-coloring. By Observation~\ref{ob:lowerboundfor2} for $k=2$, we infer that
\begin{equation*}\label{NG1}
\chi_{2}^{f}(G)\geq \Delta(G)/2+1
\end{equation*}
for any graph $G$ on at least two vertices, and we deduce that
\begin{equation}\label{NG2}
\chi_{2}^{f}(G)+\chi_{2}^{f}(\overline{G})\geq \dfrac{\Delta(G)}{2}+1+\dfrac{\Delta(\overline{G})}{2}+1=\dfrac{\Delta(G)+n-1-\delta(G)}{2}+2\geq \dfrac{n+3}{2}.
\end{equation}

Now, let $G$ be a graph on at least two vertices for which $\chi_{2}^{f}(G)+\chi_{2}^{f}(\overline{G})=(n+3)/2$. In particular, $n$ is an odd integer. Moreover, $G$ is a regular graph since the second inequality in (\ref{NG2}) necessarily holds with equality. In view of the inequality (\ref{Delta}) with $k=2$, the resulting equality $\chi_{2}^{f}(G)=\Delta(G)/2+1$ from (\ref{NG2}) implies that every vertex in $J$ has precisely two neighbors in $J'$ for each $J\in \mathcal{J}$ and $J'\in \mathcal{J}\setminus \{J\}$. Due to this, letting $\mathcal{J}=\{J_{1},\ldots,J_{\chi_{2}^{f}(G)}\}$, we infer that $|J_{1}|=\ldots=|J_{\chi_{2}^{f}(G)}|$. On the other hand, since $\mathcal{J}$ is a $2$-frugal coloring of $G$, it follows that $J_{j}$ is a clique in $\overline{G}$ for each $j\in[\chi_{2}^{f}(G)]$. Hence, each color class of any $\chi_{2}^{f}(\overline{G})$-coloring has at most one vertex from each $J_{j}$.

For the sake of simplicity, we let $t=\chi_{2}^{f}(G)$ and $r=|J_{1}|$. Because $n$ is odd, both $r$ and $t$ are odd as well. If $t=1$, then $G\cong \overline{K_{n}}$. This leads to $\chi_{2}^{f}(K_{n})+\chi_{2}^{f}(\overline{K_{n}})\geq n+1>(n+3)/2$, a contradiction. Therefore, $t\geq3$. Moreover, $r\geq3$ because $r$ is odd and every vertex in $J_{1}$ is adjacent to precisely two vertices in $J_{2}$ in the graph $G$. Suppose now that $r\geq5$. Let $\{v_{1},v_{2},v_{3}\}$ be any independent set of cardinality $3$ in $\overline{G}$. Due to the structure of $\overline{G}$, we may assume that $v_{1}\in J_{1}$, $v_{2}\in J_{2}$ and $v_{3}\in J_{3}$. Recall that in the graph $\overline{G}$, every vertex in $J_{j}$ is adjacent to precisely $r-2$ vertices of each set in $\mathcal{J}\setminus \{J_{j}\}$. With this in mind, we have $J_{2}\setminus \{v_{2},u\}\subseteq N_{\overline{G}}(v_{1})$ and $J_{2}\setminus \{v_{2},w\}\subseteq N_{\overline{G}}(v_{3})$ for some vertices $u,w\in J_{2}$. (Here, $u$ and $w$ may or may not be the same vertices.) Due to this and since $r\geq5$, it follows that there exists a vertex $x\in J_{2}$ adjacent to all $v_{1}$, $v_{2}$ and $v_{3}$ in $\overline{G}$ (recall that $J_{2}$ is a clique in $\overline{G}$). This guarantees that every color class in any $\chi_{2}^{f}(\overline{G})$-coloring has at most two vertices. Hence, $\chi_{2}^{f}(\overline{G})\geq n/2$. Therefore, $\chi_{2}^{f}(G)+\chi_{2}^{f}(\overline{G})\geq n/2+3>(n+3)/2$, which is impossible. In fact, the above argument shows that $r=3$.

\begin{figure}
\begin{center}
\begin{tikzpicture}[scale=1.5, vertex/.style={circle, fill, inner sep=2.2pt}]
\foreach \i in {1,...,9}{
\node[vertex] (v\i) at (360/9*\i:1) {};
}
\draw (0.19,0.98)--(0.75,0.67);
\draw (0.77,0.65)--(1,0);
\draw (1,0)--(0.78,-0.62);
\draw (0.78,-0.64)--(0.18,-0.99);
\draw (0.18,-0.99)--(-0.48,-0.87);
\draw (-0.49,-0.87)--(-0.92,-0.37);
\draw (-0.94,-0.37)--(-0.94,0.34);
\draw (-0.94,0.34)--(-0.48,0.9);
\draw (-0.48,0.88)--(0.19,1);
\draw (0.76,-0.62)--(0.17,0.98);
\draw (0.76,-0.64)--(-0.94,-0.35);
\draw (-0.93,-0.35)--(0.18,0.99);
\draw (0.78,0.68)--(0.18,-0.99);
\draw (1,0)--(-0.48,-0.86);
\draw (0.18,-0.99)--(-0.94,0.34);
\draw (-0.49,-0.86)--(-0.49,0.9);
\draw (-0.94,0.34)--(0.77,0.65);
\draw (-0.48,0.86)--(1,0);
\node [scale=1] at (0,-1.35) {$G_{1}$};


\fill (2.5,-0.05) circle (2pt);
\fill (2.5,0.85) circle (2pt);
\fill (2.5,-0.95) circle (2pt);

\fill (3.4,-0.05) circle (2pt);
\fill (3.4,0.85) circle (2pt);
\fill (3.4,-0.95) circle (2pt);

\fill (4.3,-0.05) circle (2pt);
\fill (4.3,0.85) circle (2pt);
\fill (4.3,-0.95) circle (2pt);

\draw (2.5,-0.05)--(2.5,0.85)--(2.5,-0.95)--(3.4,-0.95)--(4.3,-0.95)--(4.3,-0.05)--(4.3,0.85)--(3.4,0.85)--(2.5,0.85);
\draw (3.4,-0.95)--(3.4,-0.05)--(3.4,0.85);
\draw (2.5,-0.05)--(3.4,-0.05)--(4.3,-0.05);

\draw (2.5,0.88) .. controls (3.4,1.1) and (3.4,1.1) .. (4.3,0.88);
\draw (2.5,-0.02) .. controls (3.4,0.2) and (3.4,0.2) .. (4.3,-0.02);
\draw (2.5,-0.92) .. controls (3.4,-0.7) and (3.4,-0.7) .. (4.3,-0.92);
\draw (2.52,0.88) .. controls (2.72,-0.05) and (2.72,-0.05) .. (2.52,-0.95);
\draw (3.42,0.88) .. controls (3.62,-0.05) and (3.62,-0.05) .. (3.42,-0.95);
\draw (4.32,0.88) .. controls (4.52,-0.05) and (4.52,-0.05) .. (4.32,-0.95);
\node [scale=1] at (3.4,-1.35) {$G_{2}$};


\fill (5.8,0.85) circle (2pt);
\fill (5.8,-0.05) circle (2pt);
\fill (5.8,-0.95) circle (2pt);

\fill (6.7,0.85) circle (2pt);
\fill (6.7,-0.05) circle (2pt);
\fill (6.7,-0.95) circle (2pt);

\fill (7.6,0.85) circle (2pt);
\fill (7.6,-0.05) circle (2pt);
\fill (7.6,-0.95) circle (2pt);

\draw (5.8,0.85)--(5.8,-0.05)--(5.8,-0.95)--(6.7,-0.95)--(6.7,-0.05)--(6.7,0.85)--(5.8,0.85);
\draw (7.6,0.85)--(7.6,-0.05)--(7.6,-0.95)--(5.8,-0.05);
\draw (7.6,-0.95)--(6.7,0.85);
\draw (5.8,-0.95)--(7.6,-0.05)--(6.7,-0.95);
\draw (5.8,-0.05)--(6.7,-0.05)--(7.6,0.85);

\draw (5.8,0.88) .. controls (6.7,1.1) and (6.7,1.1) .. (7.6,0.88);
\draw (5.82,0.85) .. controls (5.99,-0.05) and (5.99,-0.05) .. (5.82,-0.95);
\draw (6.7,0.85) .. controls (6.92,-0.05) and (6.92,-0.05) .. (6.72,-0.95);
\draw (7.62,0.85) .. controls (7.82,-0.05) and (7.82,-0.05) .. (7.62,-0.95);
\node [scale=1] at (6.8,-1.35) {$G_{3}$};
\end{tikzpicture}
\end{center}\vspace{-6mm}
\caption{{\small For each $j\in[3]$, $\chi_{2}^{f}(G_{j})=\chi_{2}^{f}(\overline{G_{j}})=3$.}}
\label{fig:3graphs}
\end{figure}
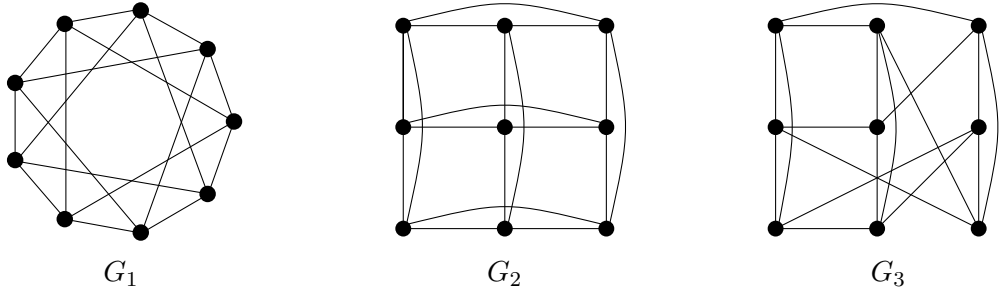

Keeping $t=\chi_{2}^{f}(G)$ and $n=3t$ in mind, we deduce from $\chi_{2}^{f}(G)+\chi_{2}^{f}(\overline{G})=(n+3)/2$ that $\chi_{2}^{f}(\overline{G})=(t+3)/2$. On the other hand, the resulting equality $\chi_{2}^{f}(\overline{G})=\Delta(\overline{G})/2+1$ from (\ref{NG2}) and interchanging $G$ with $\overline{G}$, we analogously infer that every color class in any $\chi_{2}^{f}(\overline{G})$-coloring has exactly three vertices. Due to this, the resulting equality $3\chi_{2}^{f}(\overline{G})=3(t+3)/2=n=3t$ implies that $t=3$. 

Summing up, we have proved that $G$ is a $3$-partite graph on $9$ vertices in which the subgraph induced by every two partite sets is $2$-regular. This is equivalent to saying that $\overline{G}$ is obtained from the disjoint union $K_{3}^{1}+K_{3}^{2}+K_{3}^{3}$ of triangles by adding some edges such that $[K_{3}^{r},K_{3}^{s}]$ is a matching, of cardinality $3$, for each distinct $r,s\in[3]$. Now, it is not hard to verify that, up to isomorphism, there are only three such graphs $G$ whose complements are depicted in Figure~\ref{fig:3graphs}. This contradicts the fact that $G\notin \{G_{j},\overline{G_{j}}\}_{j=1}^{3}$. This, together with the inequality (\ref{NG2}), implies that $\chi_{2}^{f}(G)+\chi_{2}^{f}(\overline{G})\geq(n+4)/2$.

That the bound is sharp, may be seen as follows. Consider the cycle $C_{n}:v_{1}v_{2}\ldots,v_{n}v_{1}$ for any even integer $n\geq4$. Clearly, $\chi_{2}^{f}(C_{n})=2$. 
On the other hand, assigning $j$ to the vertices $v_{2j-1}$ and $v_{2j}$, for each $j\in[n/2]$, gives a $2$-frugal coloring of $\overline{C_{n}}$ with $n/2$ colors.
Thus, $\chi_{2}^{f}(C_{n})+\chi_{2}^{f}(\overline{C_{n}})\le n/2+2$, which together with the proved lower bound yields the equality. This completes the proof.  
\end{proof}

In what follows, we bound $\chi_{2}^{f}(G)+\chi_{2}^{f}(\overline{G})$ from above in terms of the order of $G$, and characterize the family of extremal graphs for the bound.

\begin{theorem}\label{N-G-upper}
If $G$ is a graph of order $n\geq2$, then 
\begin{center}
$\chi_{2}^{f}(G)+\chi_{2}^{f}(\overline{G})\leq \dfrac{3n}{2}$.
\end{center} 
Moreover, equality holds if and only if $n$ is even and $G\in \{K_{1,n-1},\overline{K_{1,n-1}}\}$.
\end{theorem}
\begin{proof}
If $G\cong K_{n}$ or $G\cong \overline{K_{n}}$, then $\chi_{2}^{f}(K_n)+\chi_{2}^{f}(\overline{K_n})=n+1\leq3n/2$. Thus, we may assume that $G$ is not a complete graph and has at least one edge, in particular, $n\geq3$. 

Let $M$ be a maximum matching in $G$, and let $A$ be the subgraph of $G$ induced by the set of endvertices of edges in $M$. Furthermore, let $B$ be the subgraph induced by $V(G)\setminus V(A)$. It is clear that $|V(B)|=n-2|M|$, and that $M\neq \emptyset$ since $G$ has at least one edge. Note that vertices of $B$ form an independent set as $M$ is a maximum matching in $G$. If $n-2|M|$ is even, then we can pair the vertices of $B$ so that the vertices in each pair receive the same color in a $2$-frugal coloring of $G$, which gives the upper bound $\chi_{2}^{f}(G)\le(n-2|M|)/2+2|M|$. On the other hand, since the endvertices of an edge in $M$ can receive the same color in a $2$-frugal coloring of $\overline{G}$, we infer $\chi_{2}^{f}(\overline{G})\le|M|+n-2|M|$. Summing up, we obtain
\begin{equation}\label{even} 
\chi_{2}^{f}(G)+\chi_{2}^{f}(\overline{G})\le \dfrac{n-2|M|}{2}+2|M|+|M|+n-2|M|=\dfrac{3n}{2},
\end{equation}
as desired.

Now let $n-2|M|$ be odd. If there are two vertices $a\in V(A)$ and $b\in V(B)$ such that $ab\notin E(G)$, then $a$ and $b$ can receive the same color in a $2$-frugal coloring of $G$, while the vertices in $V(B)\setminus \{b\}$ can again be arranged in pairs. Hence, $\chi_{2}^{f}(G)\le(n-2|M|+1)/2+2|M|-1$. Using $\chi_{2}^{f}(\overline{G})\le|M|+n-2|M|$ again, we infer that
\begin{center}
$\chi_{2}^{f}(G)+\chi_{2}^{f}(\overline{G})\leq \dfrac{n-2|M|+1}{2}+2|M|-1+|M|+n-2|M|=\dfrac{3n-1}{2}$,
\end{center}
and the claimed bound holds. 
Thus, we may assume that $G=A\vee B$, that is, $G$ is the join of $A$ and $B$, where $A$ has a perfect matching and $B$ is an edgeless graph. Since $n-2|M|$ is odd, it follows that $2|M|\le n-1$. Note that $\chi_{2}^{f}(G)\le(n-2|M|-1)/2+2|M|+1$. Since $\overline{G}$ is isomorphic to the disjoint union $\overline{A}+K_{n-2|M|}$, we infer that $\chi_{2}^{f}(\overline{G})\le\max\{|M|,n-2|M|\}$. We distinguish two cases with respect to which of the values achieves the maximum. 

Firstly, assume that $|M|\ge n-2|M|$. Then, $\chi_{2}^{f}(G)+\chi_{2}^{f}(\overline{G})\le(n-2|M|-1)/2+2|M|+1+|M|=(n-1)/2+2|M|+1\le(3n-1)/2$, as desired. Secondly, let $|M|\le n-2|M|$. Then, $$\chi_{2}^{f}(G)+\chi_{2}^{f}(\overline{G})\le \frac{n-2|M|-1}{2}+2|M|+1+n-2|M|=\frac{3n+1-2|M|}{2}\le \frac{3n-1}{2},$$
where the last inequality follows from the fact that $M\neq \emptyset$. 
In all cases, we obtained the inequality $\chi_{2}^{f}(G)+\chi_{2}^{f}(\overline{G})\leq3n/2$, which proves the first statement of the theorem. 

To verify the characterization part, we first observe that $\chi_{2}^{f}(K_{1,n-1})+\chi_{2}^{f}(\overline{K_{1,n-1}})=\chi_{2}^{f}(K_{1,n-1})+\chi_{2}^{f}(K_{1}+K_{n-1})=(n/2+1)+(n-1)=3n/2$ for any even integer $n\geq2$. 
Conversely, let $\chi_{2}^{f}(G)+\chi_{2}^{f}(\overline{G})=3n/2$, $c$ be a $\chi_{2}^{f}(G)$-coloring and $c'$ be a $\chi_{2}^{f}(\overline{G})$-coloring. In particular, $n$ is necessarily even and the inequality (\ref{even}) holds with equality. This implies that $\chi_{2}^{f}(G)=(n-2|M|)/2+2|M|$ and that $\chi_{2}^{f}(\overline{G})=n-|M|$. Therefore, $A\cong K_{2|M|}$ and $B=\overline{K_{n-2|M|}}$. Trivially, $G\in \{K_{1,n-1},\overline{K_{1,n-1}}\}$ for $n=2$. So, let $n\geq3$. Suppose that $B$ is the empty graph. Then, $G\cong K_{n}$ and $\chi_{2}^{f}(G)+\chi_{2}^{f}(\overline{G})=n+1<3n/2$, a contradiction. Therefore, $B$ is not empty, and hence $|V(B)|\geq2$. 

If $|M|=1$, then we observe that $G\cong K_{3}+\overline{K_{n-3}}$ or $G\cong K_{1,r-1}+\overline{K_{n-r}}$ for some integer $r\geq2$. When the first isomorphism holds, we have $\chi_{2}^{f}(G)+\chi_{2}^{f}(\overline{G})=n+2=3n/2$. This implies that $G\cong \overline{K_{1,3}}$, and we are done. If the second isomorphism holds, then $\chi_{2}^{f}(G)+\chi_{2}^{f}(\overline{G})=(r-2)/2+2+n-1\leq3n/2$ if $r$ is even, and $\chi_{2}^{f}(G)+\chi_{2}^{f}(\overline{G})=(r-1)/2+1+n-1<3n/2$ if $r$ is odd. This implies that $n=r$, and hence $G\cong K_{1,n-1}$. If $|M|\geq2$, we then distinguish two cases depending on $[V(A),V(B)]$.\vspace{0.5mm}\\
\textit{Case 1.} $G\not\cong A\vee B$. Therefore, there exist some vertices $u\in V(A)$ and $x\in V(B)$ such that $ux\notin E(G)$. Suppose that there exists a vertex $y\in V(B)\setminus \{x\}$ which is not adjacent to some vertex $v\in V(A)\setminus \{u\}$. Then, $c''(x)=c(u)$, $c''(y)=c(v)$ and $c''(w)=c(w)$ for any other vertex $w$ defines a $2$-frugal coloring of $G$ with $|c''\big{(}V(G)\big{)}|<|c\big{(}V(G)\big{)}|$, which is impossible. This shows that every vertex in $V(B)\setminus \{x\}$ is adjacent to all vertices in $A\setminus \{u\}$. We now need to consider two more possibilities.\vspace{0.5mm}\\
\textit{Subcase 1.1.} $|V(B)|=2$. Suppose first that $xv_{1},xv_{2}\notin E(\overline{G})$ for some $v_{1},v_{2}\in V(A)\setminus \{u\}$. Let $y\in V(B)\setminus \{x\}$. In such a situation, $h(y)=c'(v_{1})=c'(v_{2})$ and $h(w)=c(w)$ for any other vertex $w$ is a $2$-frugal coloring of $\overline{G}$ with $|h\big{(}V(\overline{G})\big{)}|=n-|M|-1<|c'\big{(}V(\overline{G})\big{)}|$, which is impossible. Therefore, $|N_{\overline{G}}(x)\cap V(A)|\geq|V(A)|-1$. If $|N_{\overline{G}}(x)\cap V(A)|=|V(A)|$, then $\overline{G}\cong K_{1,n-1}$ and we are done. If $|N_{\overline{G}}(x)\cap V(A)|=|V(A)|-1$, then $\overline{G}\cong K_{1,n-2}+K_{1}$. It is then easy to check that                                                    $\chi_{2}^{f}(G)+\chi_{2}^{f}(\overline{G})=(n-2)/2+1+n-1<3n/2$, a contradiction. Thus, $G\in \{K_{1,n-1},\overline{K_{1,n-1}}\}$ when $|B|=2$.\vspace{0.5mm}\\
\textit{Subcase 1.2.} $|V(B)|\geq4$. Let $y,z\in V(B)\setminus \{x\}$ and $v_{1},v_{2}\in V(A)\setminus \{u\}$ be distinct vertices. Then, $h(v_{1})=c'(y)$, $h(v_{2})=c'(z)$ and $h(w)=c(w)$ for any other vertex $w$ defines a $2$-frugal coloring of $\overline{G}$ with $|h\big{(}V(\overline{G})\big{)}|=n-|M|-1<|c'\big{(}V(\overline{G})\big{)}|$, a contradiction.\vspace{0.5mm}\\
\textit{Case 2.} $G\cong A\vee B$. In such a case, by taking $A\cong K_{2|M|}$ and $B=\overline{K_{n-2|M|}}$ into account, it is readily seen that $\chi_{2}^{f}(G)+\chi_{2}^{f}(\overline{G})=n<3n/2$. This is a contradiction.\vspace{0.5mm}

All in all, we have proved that $G\cong K_{1,n-1}$ or $G\cong \overline{K_{1,n-1}}$, in which $n$ is even, when $\chi_{2}^{f}(G)+\chi_{2}^{f}(\overline{G})=3n/2$. This completed the proof.   
\end{proof} 


\section{Graph classes and operations}
\label{sec:classes}
\subsection{Block graphs}

Recall that a {\em block} in a graph $G$ is a maximal connected subgraph of $G$ that has no cut-vertices. A graph $G$ is a {\em block graph} if every block in $G$ is a complete graph. Applying the lower bound in Observation~\ref{ob:lowerboundfor2} for $k=2$, we infer that for any graph $G$, 
\begin{equation}\label{eq:lower}
\fd(G)\ge \max\Big\{\chi(G),\Big\lceil\frac{\Delta(G)}{2}\Big\rceil+1\Big\}.   
\end{equation}
In this section, we prove that (\ref{eq:lower}) holds with equality for block graphs. Let $\omega(G)$ stand for the {\em clique number} of $G$, which is the largest cardinality of a clique in $G$. Clearly, $\chi(G)\ge \omega(G)$ in any graph $G$, and $\chi(G)=\omega(G)$ if $G$ is a block graph. 

\begin{theorem}
\label{thm:block}
If $G$ is a block graph with maximum degree $\Delta$ and clique number $\omega$, then $$\fd(G)=\max\Big\{\omega,\Big\lceil\frac{\Delta}{2}\Big\rceil+1\Big\}.$$ 
\end{theorem}
\begin{proof}
We may assume that $G$ is connected. Consider the tree-like representation $T_G$ of the block graph $G$, in which the blocks represent vertices, and root $T_G$ at any block. Note that with respect to $T_G$, the blocks that intersect the root block $B^r$ are the children of $B^r$. In addition, any other block $B$ has exactly one parent, and all other blocks that intersect $B$ are its children. 

In the rest of the proof, we construct a $2$-frugal coloring of $G$, with $\max\{\omega,\lceil \Delta/2\rceil+1\}$ colors, which is done inductively from the root block downwards. First, we color the vertices of $B^r$ by using $|V(B^r)|$ colors. The coloring procedure is continued in a such a way that vertices of a block $B$ are colored only after the vertices of its parent block have already been colored. The proof is by induction on the number of blocks colored up to some point of this procedure. Clearly, after the base case, when $B^r$ has been colored by $|V(B^r)|$ colors, the coloring is proper and no vertex has more than two neighbors with the same color, and the number of used colors is $|V(B^r)|$, which is less than or equal to $\max\{\omega,\lceil \Delta/2\rceil+1\}$. 

Now, assume that there is a block $B_1$ such that the vertices of its parent block $B$ have already been colored by a $2$-frugal coloring $c$, and let $V(B)\cap V(B_1)=\{x\}$. Consider all blocks that intersect $B$ in $x$, denote them by $B_1,\ldots,B_\ell$, and note that they are children of $B$ in $T_G$. We will color all vertices of the blocks $B_1,\ldots,B_\ell$ simultaneously (with the clear exception of $x$ that has already been colored). We distinguish two cases with respect to which of the values in the statement of the theorem is larger. In each case, we may assume that the colors in $[|V(B)|]$ are assigned to the vertices of $B$.\vspace{1mm}\\ 
\textit{Case 1.} $\omega\ge \lceil \Delta/2\big\rceil+1$. Let $q=\omega-|V(B)|$, and let $i\in[\ell]$ be the smallest index such that $|(V(B_1)\setminus\{x\})\cup\cdots\cup (V(B_i)\setminus\{x\})|\ge q$ if such $i$ exists. In any case, color the vertices in $V(B_1)\setminus\{x\},\cdots,V(B_i)\setminus\{x\}$, respectively, with colors in $[\omega]\setminus[|V(B)|]$ as long as there are such colors. (If $q=0$, that is, if $V(B)$ is a largest clique in $G$, then this coloring step is not applicable, and we proceed to the next step.) Hence, after this is done, some vertices of $B_i$ may not have been colored. Let $B_i'$ be the set of vertices of $B_i$ that have already been colored and $B_i''$ be the set of vertices of $B_i$ that have not yet been colored up to this point. Note that $B_i''$ could be empty, while $x\in B_i'$. If $B_i''\ne\emptyset$ or if $i<\ell$, color the remaining not yet colored vertices of $B_i$ and, if applicable, also the vertices of $(V(B_{i+1})\setminus\{x\})\cup\cdots\cup (V(B_\ell)\setminus\{x\})$ by at most $\omega-1$ colors from $[\omega]\setminus \{c(x)\}$. This is possible since 
\begin{center}
$|B_i''\cup (V(B_{i+1})\setminus\{x\})\cup\cdots\cup (V(B_\ell)\setminus\{x\})|\le \Delta+1-\omega\le \omega-1$
\end{center}
(note that we assign $|B_{i}''|$ colors from $[\omega]\setminus \{c(x)\}$ to the vertices in $B_{i}''$ that have not already been assigned to the vertices in $B_{i}'$). In fact, the coloring $c$ has been extended by using at most $\omega$ colors to the vertices in $V(B_1)\cup\cdots V(B_\ell)$. The coloring is proper, and every color in $N_G(x)$ appears at most twice. There are no other vertices in $N_G(x)$ whose neighborhood needs to be verified in $N_{G}[x]$. Consequently, the proof ends by using induction.\vspace{1mm}\\ 
\textit{Case 2.} $\omega<\lceil \Delta/2\rceil+1$. If $\ell=1$, that is, if $x$ lies only in the blocks $B$ and $B_1$, then color the vertices in $V(B_1)\setminus\{x\}$ with distinct colors in $[\omega]\setminus \{c(x)\}$. Since $|V(B_1)|\le \omega<\lceil \Delta/2\rceil+1$, such a coloring of the vertices of $B_1$ is possible, and it results in a proper coloring of the vertices in $N_{G}[x]$ such that every color in the neighborhood of $x$ appears at most twice. Hence, we may assume that $\ell\ge 2$. Sequence the blocks $B_1,\ldots,B_\ell$ with respect to their orders from the largest to the smallest. Choose an arbitrary color from $S=[\lceil \Delta/2\rceil+1]\setminus[|V(B)|]$ and use it for two vertices that belong to two largest blocks, respectively. Then, update the sequence $B_1,\ldots,B_\ell$ based on the number of not yet colored vertices in each $B_{i}$. Repeat this procedure, by always choosing two blocks with the largest numbers of not yet colored vertices and color a vertex in each of these two blocks with the same color which is any color not yet used in $S$, until this is possible. If all vertices in $V(B_{1})\cup \ldots \cup V(B_{\ell})$ are colored in this way, then we are done by using induction as the resulting coloring up to this point is $2$-frugal. So, we may assume that this procedure ends before coloring all vertices in $V(B_{1})\cup \ldots \cup V(B_{\ell})$. There are two possibilities depending on when the procedure is no longer possible.\vspace{0.75mm}\\ 
\textit{Subcase 2.1}: All colors in $S$ have been exhausted. Note that each such color is used twice in $N_{G}(x)$. This, by taking the colors used in $V(B)$ into account, shows that $2\lceil \Delta/2\rceil+2-|V(B)|$ vertices of $N_{G}[x]$ have been colored so far. Therefore, at most $|V(B)|-1$ vertices of $N_{G}[x]$ remain uncolored. In such a situation, we can use $|V(B)|-1$ colors from $[|V(B)|]\setminus \{c(x)\}$ to color them. This results in a $2$-frugal coloring up to this point.\vspace{0.75mm}\\
\textit{Subcase 2.2}: Only one of the blocks remains to have uncolored vertices. By the way how we colored the vertices in $N(x)$, we may assume this block is $B_1$. Note that at least two vertices of $B_{1}$ have already been colored: the vertex $x$ and another vertex colored with a color in $S$, because $V(B_2)\setminus\{x\}\ne\emptyset$. Let $Q$ be the set of not yet colored vertices of $B_{1}$. If $|V(B)\setminus \{x\}|\geq|Q|$, then we assign $|Q|$ colors from $[|V(B)|]\setminus \{c(x)\}$ to the vertices in $Q$. It is then easily observed that the resulting coloring up to this point is $2$-frugal. 

Now, assume that $|V(B)\setminus \{x\}|<|Q|$. This in particular shows that $|V(B)|\leq|V(B_{1})|-2$. Moreover, $|V(B_{1})|=\max\{|V(B)|,|V(B_{1})|,\ldots,|V(B_{\ell})|\}$. We define the following coloring $c'$ of the subgraph $G'$ of $G$ induced by $V(B)\cup V(B_{1})\cup\ldots \cup V(B_{\ell})$. We begin with assigning the colors in $[|V(B_{1})|]$ to the vertices of $B_{1}$. Iterating the above-mentioned procedure for the blocks $B,B_{2},\ldots,B_{\ell}$, we meet one of the following possibilities:\vspace{0.25mm}\\
$(i)$ all vertices of $G'$ are colored, or\vspace{0.25mm}\\
$(ii)$ all colors not in $[|V(B_{1})|]$ are exhausted, or\vspace{0.25mm}\\
$(iii)$ only one of the blocks, say $B_{i}$, remains to have uncolored vertices.\vspace{0.25mm}

If $(i)$ holds, then all the vertices of $G'$ are colored by at most $\lceil \Delta/2\rceil+1$ colors. If $(ii)$ happens, similarly to Subcase 2.1, we complete the coloring to a $2$-frugal coloring of $G'$ with at most $\lceil \Delta/2\rceil+1$ colors. Assume now that $(iii)$ happens and that $Q_{i}'$ is the set of uncolored vertices of $B_{i}$. Since $|V(B_{1})|-1\geq|Q_{i}'|$, we can assign $|Q_{i}'|$ colors from $[|V(B_{1})|]$, different from the color of $x$, to the vertices in $Q_{i}'$. Note that the resulting coloring $c'$ of $G'$ is proper and no color appears more than twice in $N_G(x)$, that is, $c'$ is a $2$-frugal coloring of $G'$ with at most $\lceil \Delta/2\rceil+1$ colors.

Let $c'\big(V(B)\big)=\{c'_{1},\ldots,c'_{|V(B)|}\}$. We observe that any permutation of the set of colors of a $2$-frugal coloring is a $2$-frugal coloring as well. In view of this, any permutation $\sigma$ of $c'\big(V(G')\big)$ with $\sigma(c'_{1})=1,\ldots,\sigma(c'_{|V(B)|})=|V(B)|$ leads to a $2$-frugal coloring up to this point. In this way, $c$ is extended to a $2$-frugal coloring so that it colors the vertices in $V(B_{1})\cup \ldots \cup V(B_{\ell})$. Hence, also in this case, induction completes the proof. 
\end{proof}

We immediately infer the exact value of the $2$-frugal chromatic number in trees. 

\begin{corollary}
\label{cor:forest}
If $T$ is a tree with maximum degree $\Delta$, then $\fd(T)=\big\lceil\frac{\Delta}{2}\big\rceil+1$.     
\end{corollary}

\subsection{Standard graph products}

For the four standard products of graphs $G$ and $H$ (according to \cite{ImKl}), the vertex set of the product is $V(G)\times V(H)$. Their edge sets are defined as follows.
\begin{itemize}\vspace{-1.25mm}
\item In the \emph{Cartesian product} $G\square H$ two vertices are adjacent if they are adjacent in one coordinate and equal in the other.\vspace{-1.25mm}
\item Two vertices in the \emph{direct product} $G\times H$ are adjacent if they are adjacent in both coordinates.\vspace{-1.25mm}
\item The edge set of the \emph{strong product} $G\boxtimes H$ is the union of $E(G\square H)$ and $E(G\times H)$.\vspace{-1.25mm}
\item Two vertices $(g,h)$ and $(g',h')$ are adjacent in the \emph{lexicographic product} $G\circ H$ if either $gg'\in E(G)$ or ``$g=g'$ and $hh'\in E(H)$''.
\end{itemize}\vspace{-1.25mm}

\begin{theorem}
\label{thm:cartesian}
If $G$ and $H$ are arbitrary graphs, then 
$$\max\{\chi_{2}^{f}(G),\chi_{2}^{f}(H)\}\le \chi_{2}^{f}(G\cp H)\le \max\{\chi_2(G),\chi_2(H)\},$$
and the bounds are sharp.
\end{theorem}
\begin{proof}
We may assume that $k=\chi_2(G)\ge \chi_2(H)$. Let $c':V(G)\to [k]$ and $c'':V(H)\to [k]$ be $2$-distance colorings of $G$ and $H$, respectively, using (at most) $k$ colors. Define $c:V(G)\times V(H)$ as follows: $c\big{(}(g,h)\big{)}=c'(g)+c''(h) \pmod k$. Since $c'$ and $c''$ are proper colorings of $G$ and $H$, respectively, we immediately infer that $c$ is a proper coloring of $G\cp H$. Since $c'$ is a $2$-distance coloring of $G$, we infer that for every vertex $(g,h)\in V(G\cp H)$ and every two distinct vertices $(g_1,h),(g_2,h)$, where $g_1,g_2\in N_G(g)$, we have $c'(g_1)\ne c'(g_2)$. Therefore, $c'(g_1)+c''(h)\ne c'(g_2)+c''(h) \pmod k$, which yields $c(g_{1},h)\ne c(g_{2},h)$. In a similar way, using the fact that $c''$ is a $2$-distance coloring of $H$, we infer that all vertices in $\{g\}\times N_H(h)$ receive pairwise distinct colors with respect to coloring $c$. Altogether, we derive that for each color in $[k]$, there are at most two vertices in $N_{G\cp H}\big{(}(g,h)\big{)}$ that receive that color by $c$. Hence, $c$ is a $2$-frugal coloring of $G\cp H$, and so $\chi_{2}^{f}(G\cp H)\le k=\max\{\chi_2(G),\chi_2(H)\}$. 

For the sharpness of the upper bound, consider a Cartesian grid. Note that for any finite path on $n\ge3$ vertices, $\chi_2(P_n)=3$. The same holds for the two-way infinite path $\mathbb{Z}$, notably $\chi_2(\mathbb{Z})=3$. Figure~\ref{fig:grids1}(a) shows a $2$-frugal coloring of $\mathbb{Z}\cp \mathbb{Z}$ using $3$-colors, while it is clear that two colors do not suffice even for $P_n\cp P_n$ where $n\ge 3$. 

The lower bound is trivial, since a $2$-frugal coloring of $G\cp H$ restricted to a $G$-fiber (resp.\ $H$-fiber) is a $2$-frugal coloring of that fiber, which is isomorphic to $G$ (resp.\ $H$). Hence, $\chi_{2}^{f}(G\cp H)\ge \chi_{2}^{f}(G)$ and $\chi_{2}^{f}(G\cp H)\ge \chi_{2}^{f}(H)$. 
Note that $\chi_2(K_n)=\chi_{2}^{f}(K_n)$ for any $n\in \mathbb{N}$. Therefore, if $m\ge n$, we get $m=\max\{\chi_2(K_m),\chi_2(K_n)\}=\max\{\chi_{2}^{f}(K_m),\chi_{2}^{f}(K_n)\}$. As both lower and upper bounds coincide in this case, we get $\chi_{2}^{f}(K_m\cp K_n)=m$. 
\end{proof}

As an application of the upper bound in Theorem~\ref{thm:cartesian}, we present the exact value of the $2$-frugal chromatic number of torus graphs (the Cartesian products of cycles). Note that for the $2$-distance chromatic number of cycles, we have the following values:

\begin{align}\label{cycles}
\chi_2(C_n)=
\begin{cases}
3 & \mbox{if}\ n\equiv 0\!\!\!\pmod 3,\\
5 & \mbox{if}\ n=5,\\
4 & \textrm{otherwise}.
\end{cases}
\end{align}

\begin{proposition}\label{Torus}
For any integers $m,n\geq3$, 
\begin{align*}\chi_{2}^{f}(C_m\square C_n)=
\begin{cases}
3 & \mbox{if}\ m\equiv 0\!\!\!\pmod 3 \textrm{ and } n\equiv 0\!\!\!\pmod 3,\\
4 & \textrm{otherwise}.
\end{cases}
\end{align*}
\end{proposition}
\begin{proof}
For convenience, we write $T_{m,n}=C_m\square C_n$. If $m\equiv0$ (mod $3$) and $n\equiv0$ (mod $3$), then $\chi_{2}^{f}(T_{m,n})\geq\lceil(\Delta(T_{m,n})/2\rceil+1=3$, and by Theorem \ref{thm:cartesian}, we have $\chi_{2}^{f}(T_{m,n})=3$.
 
If at least one of $m$ and $n$ is congruent to $1$ modulo $3$, say $m\equiv1$ (mod $3$), we claim that $\chi_{2}^{f}(T_{m,n})\geq4$. We prove the claim by contradiction. Suppose that $f:V(T_{m,n})\rightarrow[3]$ is a $2$-frugal coloring. Let $v_{i,j}$ denote the vertex in the $i$th row and $j$th column in the $(m\times n)$-matrix form of $T_{m,n}$. Consider the $4$-cycle $v_{1,1}v_{1,n}v_{m,n}v_{m,1}v_{1,1}$. We first suppose that $f(v_{1,1})=f(v_{m,n})=1$ and $f(v_{1,n})=f(v_{m,1})=2$. Since $v_{1,2}, v_{2,1}, v_{1,n}$ and $v_{m,1}$ are adjacent to $v_{1,1}$, it follows that $f(v_{1,2})=f(v_{2,1})=3$. Similarly, $f(v_{1,n-1})=f(v_{2,n})=3$. This is a contradiction because $v_{2,1}v_{2,n}\in E(T_{m,n})$. 

Next, we suppose that $f(v_{1,1})=f(v_{m,n})=1$, $f(v_{1,n})=2$, and $f(v_{m,1})=3$. In view of this, $f(v_{1,n-1})=f(v_{2,n})=3$, and $f(v_{2,1})=2$ since $v_{2,1}v_{1,1},v_{2,1}v_{2,n}\in E(T_{m,n})$. Hence, $f(v_{1,2})=3$. In a similar fashion, we have $f(v_{m,n-1})=2$, $f(v_{m-1,n})=3$ and $f(v_{m-1,1})=f(v_{m,2})=2$. 
Note that $m>4$, for otherwise $f(v_{2,1})=f(v_{3,1})=2$, which is impossible. If $n=3$, then $f(v_{1,1})=f(v_{m,1})$, which is a contradiction (see the only two possible patterns depicted in Figure \ref{Pat1}). On the other hand, we have $n>4$ because $v_{1,2}$ and $v_{1,n-1}$ both receive color $3$. Due to this, because $v_{1,n}$ and $v_{2,1}$ are adjacent to the vertex $v_{2,n}$, it follows that $f(v_{2,n-1})=f(v_{3,n})=1$. Since $v_{3,n}v_{3,1},v_{2,1}v_{3,1}\in E(G)$, it follows that $f(v_{3,1})=3$. Because $v_{3,1}, v_{3,n-1},v_{2,n}$ and $v_{4,n}$ are the neighbors of $v_{3,n}$, we infer that $f(v_{3,n-1})=f(v_{4,n})=2$. Iterating this process, we end up with $f(v_{m-2,1})=f(v_{m-1,1})=2$ or $f(v_{m-1,1})=f(v_{m,1})=3$, which is a contradiction. Hence, $\chi_{2}^{f}(T_{m,n})\geq4$.

\begin{figure}[ht]
\centering
\begin{tikzpicture}[scale=.02, transform shape]
\node [draw, shape=circle] (v1) at (0,0) {};
\node [draw, shape=circle] (v2) at (90,0) {};
\node [draw, shape=circle] (v3) at (0,-170) {};
\node [draw, shape=circle] (v4) at (90,-170) {};

\node [scale=50] at (10,-10) {1};
\node [scale=50] at (45,-10) {2};
\node [scale=50] at (80,-10) {3};

\node [scale=50] at (10,-30) {3};
\node [scale=50] at (45,-30) {1};
\node [scale=50] at (80,-30) {2};

\node [scale=50] at (10,-50) {2};
\node [scale=50] at (45,-50) {3};
\node [scale=50] at (80,-50) {1};

\node [scale=50] at (10,-97) {1};
\node [scale=50] at (45,-97) {2};
\node [scale=50] at (80,-97) {3};

\node [scale=50] at (10,-117) {3};
\node [scale=50] at (45,-117) {1};
\node [scale=50] at (80,-117) {2};

\node [scale=50] at (10,-137) {2};
\node [scale=50] at (45,-137) {3};
\node [scale=50] at (80,-137) {1};

\draw[dashed] (0,-149)--(90,-149);

\node [scale=50] at (10,-160) {1};
\node [scale=50] at (45,-160) {2};
\node [scale=50] at (80,-160) {3};

\node [scale=50] at (45,-68) {$\vdots$};
\node [scale=50] at (10,-68) {$\vdots$};
\node [scale=50] at (80,-68) {$\vdots$};

\draw (v1)--(v2)--(v4)--(v3)--(v1);


\node [draw, shape=circle] (u1) at (160,0) {};
\node [draw, shape=circle] (u2) at (250,0) {};
\node [draw, shape=circle] (u3) at (160,-170) {};
\node [draw, shape=circle] (u4) at (250,-170) {};

\node [scale=50] at (170,-10) {1};
\node [scale=50] at (205,-10) {2};
\node [scale=50] at (240,-10) {3};

\node [scale=50] at (170,-30) {2};
\node [scale=50] at (205,-30) {3};
\node [scale=50] at (240,-30) {1};

\node [scale=50] at (170,-50) {3};
\node [scale=50] at (205,-50) {1};
\node [scale=50] at (240,-50) {2};

\node [scale=50] at (170,-97) {1};
\node [scale=50] at (205,-97) {2};
\node [scale=50] at (240,-97) {3};

\node [scale=50] at (170,-117) {2};
\node [scale=50] at (205,-117) {3};
\node [scale=50] at (240,-117) {1};

\node [scale=50] at (170,-137) {3};
\node [scale=50] at (205,-137) {1};
\node [scale=50] at (240,-137) {2};

\draw[dashed] (160,-149)--(250,-149);

\node [scale=50] at (170,-160) {1};
\node [scale=50] at (205,-160) {2};
\node [scale=50] at (240,-160) {3};

\node [scale=50] at (205,-68) {$\vdots$};
\node [scale=50] at (170,-68) {$\vdots$};
\node [scale=50] at (240,-68) {$\vdots$};

\draw (u1)--(u2)--(u4)--(u3)--(u1);

\end{tikzpicture}
\caption{\small {When $n=3$ and $m\equiv1$ (mod $3$), the torus graph $T_{m,3}$ has no $2$-frugal coloring with three colors.}}\label{Pat1}
\end{figure}

When $m\equiv2$ (mod $3$) or $n\equiv2$ (mod $3$), similar reasoning shows that $\chi_{2}^{f}(T_{m,n})\geq4$. Therefore, when $m\neq5$ and $n\neq5$, we have $\chi_{2}^{f}(T_{m,n})=4$ in view of Theorem \ref{thm:cartesian} and (\ref{cycles}). So, it remains for us to obtain the exact value of the parameter when at least one of the factors is $C_{5}$. We assume, without loss of generality, that $n=5$. Let $m=4t+j$ for some integer $t\geq0$, where $j\in \{0,1,2,3\}$. We need to consider four possibilities depending on $j$. If $t=0$, then the $(3\times5)$-pattern in Figure \ref{Pat2} gives us a $2$-frugal coloring of $T_{3,5}$ with four colors. Hence, $\chi_{2}^{f}(T_{3,5})=4$.

\begin{figure}[ht]
\centering
\begin{tikzpicture}[scale=.02, transform shape]
\node [draw, shape=circle] (v1) at (0,0) {};
\node [draw, shape=circle] (v2) at (120,0) {};
\node [draw, shape=circle] (v3) at (120,-70) {};
\node [draw, shape=circle] (v4) at (0,-70) {};

\node [scale=50] at (10,-10) {2};
\node [scale=50] at (35,-10) {4};
\node [scale=50] at (60,-10) {1};
\node [scale=50] at (85,-10) {4};
\node [scale=50] at (110,-10) {3};

\node [scale=50] at (10,-35) {3};
\node [scale=50] at (35,-35) {2};
\node [scale=50] at (60,-35) {3};
\node [scale=50] at (85,-35) {1};
\node [scale=50] at (110,-35) {4};

\node [scale=50] at (10,-60) {4};
\node [scale=50] at (35,-60) {1};
\node [scale=50] at (60,-60) {2};
\node [scale=50] at (85,-60) {3};
\node [scale=50] at (110,-60) {1};

\draw (v1)--(v2)--(v3)--(v4)--(v1);

\end{tikzpicture}
\caption{{\small An optimal $2$-frugal coloring of $T_{3,5}$.}}\label{Pat2}
\end{figure} 
 Therefore, we may assume that $t\geq1$, and we consider four possibilities. If
\begin{itemize}\vspace{-2.5mm}
\item $j=0$, then $t$ copies of the pattern $\mathsf{A}$ in Figure \ref{Pat2} provides a $2$-frugal coloring of $T_{4t,5}$ with four colors,\vspace{-2.5mm}
\item $j=1$, then $t-1$ successive copies of $\mathsf{A}$ along with one copy of the pattern $\mathsf{B}$ represent a $2$-frugal coloring of $T_{4t+1,5}$,\vspace{-2.5mm}
\item $j=2$, then $t-1$ successive copies of $\mathsf{C}$ and one copy of $\mathsf{D}$ give a $2$-frugal coloring of $T_{4t+2,5}$, and\vspace{-2.5mm}
\item $j=3$, then $t-1$ successive copies of $\mathsf{E}$ and one copy of the pattern $\mathsf{F}$ yield a $2$-frugal coloring of $T_{4t+3,5}$.
\end{itemize}\vspace{-2.5mm}

\begin{figure}[ht]
\centering
\begin{tikzpicture}[scale=.02, transform shape]
\node [draw, shape=circle] (v1) at (-310,0) {};
\node [draw, shape=circle] (v2) at (-190,0) {};
\node [draw, shape=circle] (v3) at (-310,-95) {};
\node [draw, shape=circle] (v4) at (-190,-95) {};

\node [scale=50] at (-300,-10) {1};
\node [scale=50] at (-275,-10) {2};
\node [scale=50] at (-250,-10) {1};
\node [scale=50] at (-225,-10) {2};
\node [scale=50] at (-200,-10) {3};

\node [scale=50] at (-300,-35) {3};
\node [scale=50] at (-275,-35) {4};
\node [scale=50] at (-250,-35) {3};
\node [scale=50] at (-225,-35) {4};
\node [scale=50] at (-200,-35) {1};

\node [scale=50] at (-300,-60) {2};
\node [scale=50] at (-275,-60) {1};
\node [scale=50] at (-250,-60) {2};
\node [scale=50] at (-225,-60) {1};
\node [scale=50] at (-200,-60) {4};

\node [scale=50] at (-300,-85) {4};
\node [scale=50] at (-275,-85) {3};
\node [scale=50] at (-250,-85) {4};
\node [scale=50] at (-225,-85) {3};
\node [scale=50] at (-200,-85) {2};

\draw (v1)--(v2)--(v4)--(v3)--(v1);
\node [scale=53] at (-250,-110) {$\mathsf{A}$};


\node [draw, shape=circle] (w1) at (-165,0) {};
\node [draw, shape=circle] (w2) at (-45,0) {};
\node [draw, shape=circle] (w3) at (-165,-120) {};
\node [draw, shape=circle] (w4) at (-45,-120) {};

\node [scale=50] at (-155,-10) {1};
\node [scale=50] at (-130,-10) {2};
\node [scale=50] at (-105,-10) {1};
\node [scale=50] at (-80,-10) {2};
\node [scale=50] at (-55,-10) {3};

\node [scale=50] at (-155,-35) {3};
\node [scale=50] at (-130,-35) {4};
\node [scale=50] at (-105,-35) {3};
\node [scale=50] at (-80,-35) {4};
\node [scale=50] at (-55,-35) {1};

\node [scale=50] at (-155,-60) {2};
\node [scale=50] at (-130,-60) {1};
\node [scale=50] at (-105,-60) {2};
\node [scale=50] at (-80,-60) {1};
\node [scale=50] at (-55,-60) {4};

\node [scale=50] at (-155,-85) {1};
\node [scale=50] at (-130,-85) {4};
\node [scale=50] at (-105,-85) {3};
\node [scale=50] at (-80,-85) {2};
\node [scale=50] at (-55,-85) {3};

\node [scale=50] at (-155,-110) {2};
\node [scale=50] at (-130,-110) {3};
\node [scale=50] at (-105,-110) {4};
\node [scale=50] at (-80,-110) {1};
\node [scale=50] at (-55,-110) {4};

\draw (w1)--(w2)--(w4)--(w3)--(w1);
\node [scale=53] at (-105,-135) {$\mathsf{B}$};


\node [draw, shape=circle] (x1) at (-20,0) {};
\node [draw, shape=circle] (x2) at (100,0) {};
\node [draw, shape=circle] (x3) at (-20,-95) {};
\node [draw, shape=circle] (x4) at (100,-95) {};

\node [scale=50] at (-10,-10) {1};
\node [scale=50] at (15,-10) {2};
\node [scale=50] at (40,-10) {3};
\node [scale=50] at (65,-10) {4};
\node [scale=50] at (90,-10) {3};

\node [scale=50] at (-10,-35) {3};
\node [scale=50] at (15,-35) {4};
\node [scale=50] at (40,-35) {1};
\node [scale=50] at (65,-35) {2};
\node [scale=50] at (90,-35) {1};

\node [scale=50] at (-10,-60) {2};
\node [scale=50] at (15,-60) {1};
\node [scale=50] at (40,-60) {4};
\node [scale=50] at (65,-60) {3};
\node [scale=50] at (90,-60) {4};

\node [scale=50] at (-10,-85) {4};
\node [scale=50] at (15,-85) {3};
\node [scale=50] at (40,-85) {2};
\node [scale=50] at (65,-85) {1};
\node [scale=50] at (90,-85) {2};

\draw (x1)--(x2)--(x4)--(x3)--(x1);
\node [scale=53] at (40,-110) {$\mathsf{C}$};


\node [draw, shape=circle] (y1) at (125,0) {};
\node [draw, shape=circle] (y2) at (245,0) {};
\node [draw, shape=circle] (y3) at (125,-145) {};
\node [draw, shape=circle] (y4) at (245,-145) {};

\node [scale=50] at (135,-10) {1};
\node [scale=50] at (160,-10) {2};
\node [scale=50] at (185,-10) {3};
\node [scale=50] at (210,-10) {4};
\node [scale=50] at (235,-10) {3};

\node [scale=50] at (135,-35) {3};
\node [scale=50] at (160,-35) {4};
\node [scale=50] at (185,-35) {1};
\node [scale=50] at (210,-35) {2};
\node [scale=50] at (235,-35) {1};

\node [scale=50] at (135,-60) {2};
\node [scale=50] at (160,-60) {1};
\node [scale=50] at (185,-60) {4};
\node [scale=50] at (210,-60) {3};
\node [scale=50] at (235,-60) {4};

\node [scale=50] at (135,-85) {4};
\node [scale=50] at (160,-85) {3};
\node [scale=50] at (185,-85) {2};
\node [scale=50] at (210,-85) {1};
\node [scale=50] at (235,-85) {2};

\node [scale=50] at (135,-110) {3};
\node [scale=50] at (160,-110) {1};
\node [scale=50] at (185,-110) {4};
\node [scale=50] at (210,-110) {3};
\node [scale=50] at (235,-110) {1};

\node [scale=50] at (135,-135) {2};
\node [scale=50] at (160,-135) {4};
\node [scale=50] at (185,-135) {1};
\node [scale=50] at (210,-135) {2};
\node [scale=50] at (235,-135) {4};

\draw (y1)--(y2)--(y4)--(y3)--(y1);
\node [scale=53] at (185,-160) {$\mathsf{D}$};


\node [draw, shape=circle] (z1) at (270,0) {};
\node [draw, shape=circle] (z2) at (390,0) {};
\node [draw, shape=circle] (z3) at (270,-95) {};
\node [draw, shape=circle] (z4) at (390,-95) {};

\node [scale=50] at (280,-10) {1};
\node [scale=50] at (305,-10) {2};
\node [scale=50] at (330,-10) {3};
\node [scale=50] at (355,-10) {4};
\node [scale=50] at (380,-10) {3};

\node [scale=50] at (280,-35) {3};
\node [scale=50] at (305,-35) {4};
\node [scale=50] at (330,-35) {1};
\node [scale=50] at (355,-35) {2};
\node [scale=50] at (380,-35) {4};

\node [scale=50] at (280,-60) {2};
\node [scale=50] at (305,-60) {1};
\node [scale=50] at (330,-60) {4};
\node [scale=50] at (355,-60) {3};
\node [scale=50] at (380,-60) {1};

\node [scale=50] at (280,-85) {4};
\node [scale=50] at (305,-85) {3};
\node [scale=50] at (330,-85) {2};
\node [scale=50] at (355,-85) {1};
\node [scale=50] at (380,-85) {2};

\draw (z1)--(z2)--(z4)--(z3)--(z1);
\node [scale=53] at (330,-110) {$\mathsf{E}$};


\node [draw, shape=circle] (o1) at (415,0) {};
\node [draw, shape=circle] (o2) at (535,0) {};
\node [draw, shape=circle] (o3) at (415,-170) {};
\node [draw, shape=circle] (o4) at (535,-170) {};

\node [scale=50] at (425,-10) {1};
\node [scale=50] at (450,-10) {2};
\node [scale=50] at (475,-10) {3};
\node [scale=50] at (500,-10) {4};
\node [scale=50] at (525,-10) {3};

\node [scale=50] at (425,-35) {3};
\node [scale=50] at (450,-35) {4};
\node [scale=50] at (475,-35) {1};
\node [scale=50] at (500,-35) {2};
\node [scale=50] at (525,-35) {4};

\node [scale=50] at (425,-60) {2};
\node [scale=50] at (450,-60) {1};
\node [scale=50] at (475,-60) {4};
\node [scale=50] at (500,-60) {3};
\node [scale=50] at (525,-60) {1};

\node [scale=50] at (425,-85) {4};
\node [scale=50] at (450,-85) {3};
\node [scale=50] at (475,-85) {2};
\node [scale=50] at (500,-85) {1};
\node [scale=50] at (525,-85) {2};

\node [scale=50] at (425,-110) {3};
\node [scale=50] at (450,-110) {2};
\node [scale=50] at (475,-110) {1};
\node [scale=50] at (500,-110) {3};
\node [scale=50] at (525,-110) {4};

\node [scale=50] at (425,-135) {2};
\node [scale=50] at (450,-135) {4};
\node [scale=50] at (475,-135) {3};
\node [scale=50] at (500,-135) {4};
\node [scale=50] at (525,-135) {1};

\node [scale=50] at (425,-160) {4};
\node [scale=50] at (450,-160) {1};
\node [scale=50] at (475,-160) {2};
\node [scale=50] at (500,-160) {1};
\node [scale=50] at (525,-160) {2};

\draw (o1)--(o2)--(o4)--(o3)--(o1);
\node [scale=53] at (475,-185) {$\mathsf{F}$};

\end{tikzpicture}\vspace{-8mm}
\caption{{\small Patterns used for exhibiting an optimal $2$-frugal coloring of $T_{m,5}$ for $m\geq4$.}}\label{Pat3}
\end{figure}

In either case we have shown that there exists a $2$-frugal coloring of $T_{m,5}$ with four colors for each integer $m\geq3$, Thus, $\chi_{2}^{f}(T_{m,5})=4$ for all integers $m\geq3$. 
\end{proof}

The following upper bound on the $2$-frugal chromatic number of the strong (resp.~direct) product of two graphs uses an optimal $2$-frugal coloring of one factor and an optimal $2$-distance (resp.~injective) coloring of the other factor. Recall that a function $f:V(G)\rightarrow[k]$ is an {\em injective $k$-coloring} if no vertex $v$ is adjacent to two vertices $u$ and $w$ with $f(u)=f(w)$. The minimum $k$ for which a graph $G$ admits an injective $k$-coloring is the {\em injective chromatic number} of $G$, denoted by $\chi_{i}(G)$. The study of this concept was initiated in \cite{hkss} (see also \cite{bsy} and the references therein).

\begin{theorem}\label{thm:strong-direct}
Let $G$ and $H$ be arbitrary connected graphs. Then,\vspace{1mm}\\ 
$(i)$ $\chi_{2}^{f}(G\boxtimes H)\le \min\{\chi_2(G)\chi_{2}^{f}(H),\chi_2(H)\chi_{2}^{f}(G)\}$ and\vspace{1mm}\\
$(ii)$ $\chi_{2}^{f}(G\times H)\le \min\{\chi_{i}(G)\chi_{2}^{f}(H),\chi_{i}(H)\chi_{2}^{f}(G)\}$.\vspace{1mm}\\
These bounds are sharp.
\end{theorem}
\begin{proof}
Let $c':V(G)\to [\chi_2(G)]$ be a 2-distance coloring of $G$ and $c'':V(H)\to[\chi_{2}^{f}(H)]$ be a $2$-frugal coloring of $H$. We define $c:V(G)\times V(H)\to[\chi_2(G)]\times[\chi_{2}^{f}(H)]$ by $c(g,h)=\big{(}c'(g),c''(h)\big{)}$. Evidently, $c$ is a proper coloring of $G\boxtimes H$ as $c'$ and $c''$ are proper colorings in $G$ and $H$, respectively. Suppose that there exists a vertex $(g,h)$ adjacent to distinct vertices $(g_{1},h_{1})$, $(g_{2},h_{2})$ and $(g_{3},h_{3})$ such that $c(g_{1},h_{1})=c(g_{2},h_{2})=c(g_{3},h_{3})$. In particular, $g\in N_{G}[g_{1}]\cap N_{G}[g_{2}]\cap N_{G}[g_{3}]$. Since $c'$ is a $2$-distance coloring of $G$, it necessarily follows that $g_{1}=g_{2}=g_{3}$. Due to this and the fact that $c''$ is a proper coloring of $H$, we deduce that $h\in N_{H}(h_{1})\cap N_{H}(h_{2})\cap N_{H}(h_{3})$. This is a contradiction as $c''$ is a $2$-frugal coloring of $H$. Therefore, $c$ is a $2$-frugal coloring of $G\boxtimes H$, and hence $\chi_{2}^{f}(G\boxtimes H)\leq|c\big{(}V(G)\times V(H)\big{)}|=\chi_2(G)\chi_{2}^{f}(H)$. This implies the upper bound $(i)$ as the strong product is commutative.

The upper bound in $(ii)$ can be proved in a similar way. Notably, the definition of an appropriate $2$-frugal coloring of $G\times H$ can be obtained by modifying the definition of coloring $c$ from the previous paragraph in such a way that $c'$ presents an injective coloring of $G$.

To see the upper bound $(i)$ is sharp, we consider the graph $C_{m}\boxtimes C_{n}$ for $m,n\geq3$ when $m\equiv0$ (mod $3$), $m\not\equiv0$ (mod $5$), $n\equiv0$ (mod $2$) and $n\not\equiv0$ (mod $5$). The upper bound in this situation leads to $\chi_{2}^{f}(C_{m}\boxtimes C_{n})\leq \chi_2(C_{m})\chi_{2}^{f}(C_{n})=6$. On the other hand, $\chi_{2}^{f}(C_{m}\boxtimes C_{n})\geq \Delta(C_{m}\boxtimes C_{n})/2+1=5$. Suppose that $\chi_{2}^{f}(C_{m}\boxtimes C_{n})=5$. It is then not hard to see that all color classes have the same cardinality. This contradicts the fact that neither of $m$ and $n$ is congruent to $0$ modulo $5$. Therefore, $\chi_{2}^{f}(C_{m}\boxtimes C_{n})=6$, which is equal to the upper bound in this case.
 
The sharpness of the upper bound $(ii)$ can be verified by considering the graph $C_{m}\times C_{n}$ for $m,n\geq3$ where $m\equiv0$ (mod $4$), $m\not\equiv0$ (mod $3$), $n\equiv0$ (mod $2$) and $n\not\equiv0$ (mod $3$). Then, we have $\chi_{2}^{f}(C_{m}\times C_{n})\leq \chi_{i}(C_{m})\chi_{2}^{f}(C_{n})=4$ by the upper bound. Moreover, $\chi_{2}^{f}(C_{m}\times C_{n})\geq \Delta(C_{m}\times C_{n})/2+1=3$. If $\chi_{2}^{f}(C_{m}\times C_{n})=3$, then all color classes have the same cardinality. This is impossible due to our choices of $m$ and $n$. Therefore, $\chi_{2}^{f}(C_{m}\times C_{n})=4$, which coincides with the upper bound in this case.
\end{proof}

\begin{theorem}\label{Lexi}
For any graphs $G$ and $H$, 
\begin{center}
$\chi_{2}^{f}(H)+\Big{\lceil}\frac{\Delta(G)|V(H)|}{2}\Big{\rceil}\leq \chi_{2}^{f}(G\circ H)\leq \chi_{2}^{f}(G)|V(H)|$.
\end{center} 
These bounds are sharp.
\end{theorem}
\begin{proof}
Let $\{B_{1},\ldots,B_{\chi_{2}^{f}(G)}\}$ be a $\chi_{2}^{f}(G)$-coloring and let $V(H)=\{h_{1},\ldots,h_{|V(H)|}\}$. Let $f:V(G\circ H)\to [\chi_{2}^{f}(G)|V(H)|]$ be defined by 
$f(g,h)=(t-1)n+j$, where $g\in B_t$ and $h=h_j$. 
Note that $B_{t}\times \{h_{j}\}$, for $t\in[\chi_{2}^{f}(G)]$ and $j\in[|V(H)|]$, are the color classes of $f$. 
Since $B_{t}$ is an I$2$F set in $G$, the adjacency role of the lexicographic product graphs shows that every color class $B_{t}\times \{h_{j}\}$ is an independent set in $G\circ H$. Suppose that there exists a vertex $(g,h)\in V(G\circ H)$ adjacent to three vertices $(g_{1},h_{j})$, $(g_{2},h_{j})$ and $(g_{3},h_{j})$ in $B_{t}\times \{h_j\}$ for some $t\in[\chi_{2}^{f}(G)]$ and $j\in[|V(H)|]$. This in particular implies that $g_{1},g_{2},g_{3}\in N_{G}[g]$. Since $B_{t}$ is an $2$FI set in $G$, we may assume that $g=g_{1}$ and $g_{2},g_{3}\in N_{G}(g)$. This is impossible because $B_{t}$ is an independent set in $G$. Therefore, $B_{t}\times \{h_{j}\}$ is an I$2$F set in $G\circ H$ for each $t\in[\chi_{2}^{f}(G)]$ and $j\in[|V(H)|]$. Thus, $\chi_{2}^{f}(G\circ H)\leq|f\big{(}V(G\circ H)\big{)}|=\chi_{2}^{f}(G)|V(H)|$.

Let $g$ be a vertex of maximum degree in $G$. Let $c$ be a $\chi_{2}^{f}(G\circ H)$-coloring. Clearly, $c$ assigns at least $\chi_{2}^{f}(H)$ colors to the vertices in $\{g\}\times V(H)$. The adjacency role of $G\circ H$ shows that every vertex in $\{g\}\times V(H)$ is adjacent to all vertices in $N_{G}(g)\times V(H)$. Therefore, $c\big{(}\{g\}\times V(H)\big{)}\cap c\big{(}N_{G}(g)\times V(H)\big{)}=\emptyset$. Moreover, no color class of $c$ has at least three vertices from $N_{G}(g)\times V(H)$. This implies that $c$ assigns at least $\lceil(\Delta(G)|V(H)|)/2\rceil$ colors to the vertices in $N_{G}(g)\times V(H)$. Thus, 
\begin{center}
$\chi_{2}^{f}(G\circ H)=|c\big{(}V(G\circ H)\big{)}|\geq|c\big{(}\{g\}\times V(H)\big{)}|+|c\big{(}N_{G}(g)\times V(H)\big{)}|\geq \chi_{2}^{f}(H)+\lceil(\Delta(G)|V(H)|)/2\rceil$.
\end{center}

The sharpness of the bounds can be verified by taking any graph $G\in \Psi_{2}$ (see Section~\ref{sec:generalbounds} for its definition) and $H\cong K_{m}$ for any positive integer $m$. Recall that $G$ is an $r$-partite graph such that for each partite set $X$ and $g\in X$, the vertex $g$ has precisely two neighbors in every other partite set. Invoking the proof of Theorem \ref{General1} with $k=2$, we have $\chi_{2}^{f}(G)=r=\Delta(G)/2+1$. Now, by taking $\chi_{2}^{f}(K_{m})=m$ into account, both lower and upper bounds equal $(\Delta(G)/2+1)m$. This completes the proof.
\end{proof}

\subsection{Graphs attaining the basic lower bound}

Invoking~\eqref{eq:lower} again, which is obtained from the basic lower bound in Observation~\ref{ob:lowerboundfor2} for $t=2$, we are interested in the question of which graphs $G$ attain the lower bound, that is, $$\fd(G)=\lceil\Delta(G)/2\rceil+1.$$ Note that this value intrinsically distinguishes $2$-frugal coloring from the standard coloring. For instance, by Corollary~\ref{cor:forest}, every forest $F$ satisfies this equality while $\chi(F)=2$. 

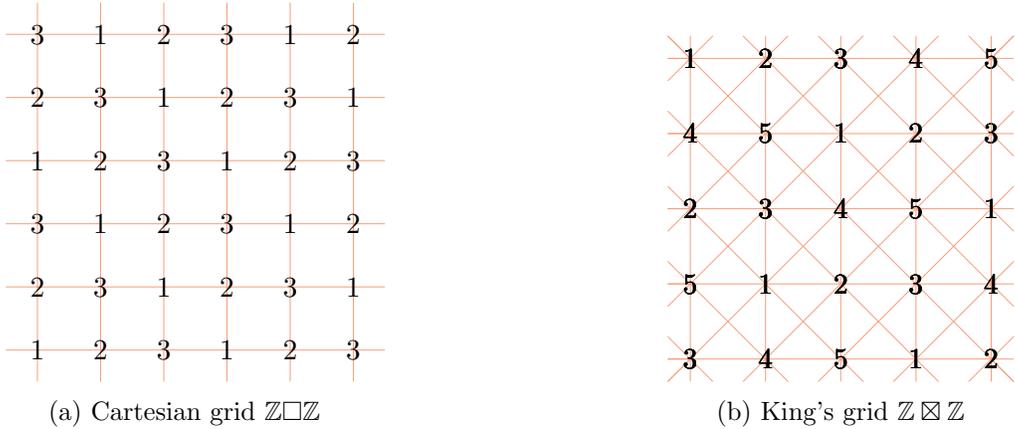
\begin{figure}[!ht]
\centering
	\begin{subfigure}[t]{0.4\textwidth}
		\centering
\begin{tikzpicture}[scale=0.42, line width=0.2pt]


  \draw[step=2cm,Melon,very thin] (-1,-1) grid (11,11);

  \draw (0,0) node {$1$};
  \draw (0,2) node {$2$};
    \draw (0,4) node {$3$};
  \draw (0,6) node {$1$};
  \draw (0,8) node {$2$};
    \draw (0,10) node {$3$};
  \draw (2,0) node {$2$};
  \draw (2,2) node {$3$};
    \draw (2,4) node {$1$};
  \draw (2,6) node {$2$};
  \draw (2,8) node {$3$};
    \draw (2,10) node {$1$};
     \draw (4,0) node {$3$};
  \draw (4,2) node {$1$};
    \draw (4,4) node {$2$};
  \draw (4,6) node {$3$};
  \draw (4,8) node {$1$};
    \draw (4,10) node {$2$};
      \draw (6,0) node {$1$};
  \draw (6,2) node {$2$};
    \draw (6,4) node {$3$};
  \draw (6,6) node {$1$};
  \draw (6,8) node {$2$};
    \draw (6,10) node {$3$};
      \draw (8,0) node {$2$};
  \draw (8,2) node {$3$};
    \draw (8,4) node {$1$};
  \draw (8,6) node {$2$};
  \draw (8,8) node {$3$};
    \draw (8,10) node {$1$};
     \draw (10,0) node {$3$};
  \draw (10,2) node {$1$};
    \draw (10,4) node {$2$};
  \draw (10,6) node {$3$};
  \draw (10,8) node {$1$};
    \draw (10,10) node {$2$};
\end{tikzpicture}
		\caption{Cartesian grid $\mathbb{Z}\square \mathbb{Z}$ \,\,\,  }
	\label{key2}
\end{subfigure} \hspace{6mm}
\begin{subfigure}[t]{0.48\textwidth}
	\centering
\begin{tikzpicture}[line width=0.2pt]
\tikzset{vertex/.style={inner sep=0pt, minimum size=0pt} 
}
	
	\def\m{5}
	\def\n{5}
	\def\dx{1}
	\def\dy{1}
	
	\foreach \i in {1,...,\m} {
		\foreach \j in {1,...,\n} {
			\node[vertex] (v\i\j) at ({(\i-1)*\dx},{(\j-1)*\dy}) {};
		}
	}

\draw[Melon,very thin] (0,0) -- (0,-0.3);
\draw[Melon,very thin] (1,0) -- (1,-0.3);
\draw[Melon,very thin] (2,0) -- (2,-0.3);
\draw[Melon,very thin] (3,0) -- (3,-0.3);
\draw[Melon,very thin] (4,0) -- (4,-0.3);

\draw[Melon,very thin] (0,4) -- (0,4.3);
\draw[Melon,very thin] (1,4) -- (1,4.3);
\draw[Melon,very thin] (2,4) -- (2,4.3);
\draw[Melon,very thin] (3,4) -- (3,4.3);
\draw[Melon,very thin] (4,4) -- (4,4.3);

\draw[Melon,very thin] (4,0) -- (4.3,0);
\draw[Melon,very thin] (4,1) -- (4.3,1);
\draw[Melon,very thin] (4,2) -- (4.3,2);
\draw[Melon,very thin] (4,3) -- (4.3,3);
\draw[Melon,very thin] (4,4) -- (4.3,4);

\draw[Melon,very thin] (0,0) -- (-0.3,0);
\draw[Melon,very thin] (0,1) -- (-0.3,1);
\draw[Melon,very thin] (0,2) -- (-0.3,2);
\draw[Melon,very thin] (0,3) -- (-0.3,3);
\draw[Melon,very thin] (0,4) -- (-0.3,4);

\draw[Melon,very thin] (0,0) -- (-0.3,0);
\draw[Melon,very thin] (0,1) -- (-0.3,1);
\draw[Melon,very thin] (0,2) -- (-0.3,2);
\draw[Melon,very thin] (0,3) -- (-0.3,3);
\draw[Melon,very thin] (0,4) -- (-0.3,4);

\draw[Melon,very thin] (0,0) -- (-0.3,0.3);
\draw[Melon,very thin] (0,1) -- (-0.3,1.3);
\draw[Melon,very thin] (0,2) -- (-0.3,2.3);
\draw[Melon,very thin] (0,3) -- (-0.3,3.3);
\draw[Melon,very thin] (0,4) -- (-0.3,4.3);

\draw[Melon,very thin] (0,0) -- (-0.3,-0.3);
\draw[Melon,very thin] (0,1) -- (-0.3,0.7);
\draw[Melon,very thin] (0,2) -- (-0.3,1.7);
\draw[Melon,very thin] (0,3) -- (-0.3,2.7);
\draw[Melon,very thin] (0,4) -- (-0.3,3.7);

\draw[Melon,very thin] (4,0) -- (4.3,0.3);
\draw[Melon,very thin] (4,1) -- (4.3,1.3);
\draw[Melon,very thin] (4,2) -- (4.3,2.3);
\draw[Melon,very thin] (4,3) -- (4.3,3.3);
\draw[Melon,very thin] (4,4) -- (4.3,4.3);

\draw[Melon,very thin] (4,0) -- (4.3,-0.3);
\draw[Melon,very thin] (4,1) -- (4.3,0.7);
\draw[Melon,very thin] (4,2) -- (4.3,1.7);
\draw[Melon,very thin] (4,3) -- (4.3,2.7);
\draw[Melon,very thin] (4,4) -- (4.3,3.7);

\draw[Melon,very thin] (0,4) -- (0.3,4.3);
\draw[Melon,very thin] (1,4) -- (1.3,4.3);
\draw[Melon,very thin] (2,4) -- (2.3,4.3);
\draw[Melon,very thin] (3,4) -- (3.3,4.3);
\draw[Melon,very thin] (4,4) -- (4.3,4.3);

\draw[Melon,very thin] (0,4) -- (-0.3,4.3);
\draw[Melon,very thin] (1,4) -- (0.7,4.3);
\draw[Melon,very thin] (2,4) -- (1.7,4.3);
\draw[Melon,very thin] (3,4) -- (2.7,4.3);
\draw[Melon,very thin] (4,4) -- (3.7,4.3);

\draw[Melon,very thin] (0,0) -- (0.3,-0.3);
\draw[Melon,very thin] (1,0) -- (1.3,-0.3);
\draw[Melon,very thin] (2,0) -- (2.3,-0.3);
\draw[Melon,very thin] (3,0) -- (3.3,-0.3);
\draw[Melon,very thin] (4,0) -- (4.3,-0.3);

\draw[Melon,very thin] (0,0) -- (-0.3,-0.3);
\draw[Melon,very thin] (1,0) -- (0.7,-0.3);
\draw[Melon,very thin] (2,0) -- (1.7,-0.3);
\draw[Melon,very thin] (3,0) -- (2.7,-0.3);
\draw[Melon,very thin] (4,0) -- (3.7,-0.3);

	\foreach \i in {1,...,\m} {
		\foreach \j in {1,...,\n} {
			
			\pgfmathtruncatemacro{\ip}{\i+1}
			\pgfmathtruncatemacro{\jm}{\j-1}
			\pgfmathtruncatemacro{\jp}{\j+1}
			
			\ifnum\ip>\m\else
			\draw[Melon,very thin] (v\i\j)--(v\ip\j);
			\fi
			
			\ifnum\jp>\n\else
			\draw[Melon,very thin] (v\i\j)--(v\i\jp);
			\fi
			
			\ifnum\ip>\m\else
			\ifnum\jp>\n\else
			\draw[Melon,very thin] (v\i\j)--(v\ip\jp);
			\fi\fi
			
			\ifnum\ip>\m\else
			\ifnum\jm<1\else
			\draw[Melon,very thin] (v\i\j)--(v\ip\jm);
			\fi\fi
		}
\draw (0,0) node {$3$};
\draw (1,0) node{$4$};
\draw (2,0) node {$5$};
\draw (3,0) node {$1$};
\draw (4,0) node{$2$};

\draw (0,1) node {$5$};
\draw (1,1) node{$1$};
\draw (2,1) node{$2$};
\draw (3,1) node{$3$};
\draw (4,1) node {$4$};

\draw (0,2) node{$2$};
\draw (1,2) node{$3$};
\draw (2,2) node{$4$};
\draw (3,2) node{$5$};
\draw (4,2) node{$1$};

\draw (0,3) node {$4$};
\draw (1,3) node {$5$};
\draw (2,3) node {$1$};
\draw (3,3) node {$2$};
\draw (4,3) node {$3$};

\draw (0,4) node{$1$};
\draw (1,4) node{$2$};
\draw (2,4) node{$3$};
\draw (3,4) node{$4$};
\draw (4,4) node{$5$};

	}
\end{tikzpicture}
\caption{King's grid $\mathbb{Z}\boxtimes \mathbb{Z}$}
\end{subfigure}
\caption{$\fd$-colorings of the Cartesian and King's grids}
\label{fig:grids1}
\end{figure}

\begin{figure}[ht]
	\centering
	\begin{subfigure}[t]{0.48\textwidth}
		\centering
	\begin{tikzpicture}[scale=0.370, line width=0.2pt]
		
		\def\a{1.5}
		
		\pgfmathsetmacro{\w}{2*\a}
		\pgfmathsetmacro{\h}{\a*sqrt(3)}
		
		\pgfmathsetmacro{\dx}{1.5*\a}   
		\pgfmathsetmacro{\dy}{\h}       
		\pgfmathsetmacro{\shift}{\h/2}  
		
		\newcommand{\hex}[2]{
			\draw [Melon,very thin]
			(#1 + 0       , #2 + \h/2 )
			-- (#1 + \a/2   , #2 + \h   )
			-- (#1 + 3*\a/2 , #2 + \h   )
			-- (#1 + 2*\a   , #2 + \h/2 )
			-- (#1 + 3*\a/2 , #2        )
			-- (#1 + \a/2   , #2        )
			-- cycle;
		}

		\foreach \x in {0,...,4}{
			\foreach \y in {0,...,3}{
				\pgfmathparse{mod(\x,2)==1 ? \shift : 0}
				\edef\yshift{\pgfmathresult}
				\hex{\x*\dx}{\y*\dy - \yshift}
			}
		}
		
		\draw (3,-1.299) node {$2$};
		\draw (4.5,-1.299) node {$1$};
		\draw (7.5,-1.299) node {$2$};
		\draw (9,-1.299) node {$1$};
		
		\draw (0.75,0) node {$3$};
		\draw (2.25,0) node {$1$};
		\draw (5.25,0) node {$3$};
		\draw (6.75,0) node {$1$};
		\draw (9.75,0) node {$3$};
		\draw (11.25,0) node {$1$};
		
		\draw (0,1.299) node {$2$};
		\draw (3,1.299) node {$3$};
		\draw (4.5,1.299) node {$2$};
		\draw (7.5,1.299) node {$3$};
		\draw (9.1,1.299) node {$2$};
		\draw (12,1.299) node {$3$};

		\draw (0.75,2.598) node {$1$};
		\draw (2.25,2.598) node {$2$};
		\draw (5.25,2.598) node {$1$};
		\draw (6.75,2.598) node {$2$};
		\draw (9.75,2.598) node {$1$};
		\draw (11.25,2.598) node {$2$};

		\draw (0,3.8978) node {$3$};
		\draw (3,3.8978) node {$1$};
		\draw (4.5,3.8978) node {$3$};
		\draw (7.5,3.8978) node {$1$};
		\draw (9.1,3.8978) node {$3$};
		\draw (12,3.8978) node {$1$};	
		
		\draw (0.75,5.1986) node {$2$};
		\draw (2.25,5.1986) node {$3$};
		\draw (5.25,5.1986) node {$2$};
		\draw (6.75,5.1986) node {$3$};
		\draw (9.75,5.1986) node {$2$};
		\draw (11.25,5.1986) node {$3$};

		\draw (0,6.4958) node {$1$};
		\draw (3,6.4958) node {$2$};
		\draw (4.5,6.4958) node {$1$};
		\draw (7.5,6.4958) node {$2$};
		\draw (9.1,6.4958) node {$1$};
		\draw (12,6.4958) node {$2$};

		\draw (0.75,7.7948) node {$3$};
		\draw (2.25,7.7948) node {$1$};
		\draw (5.25,7.7948) node {$3$};
		\draw (6.75,7.7948) node {$1$};
		\draw (9.75,7.7948) node {$3$};	
		\draw (11.25,7.7948) node {$1$};

		\draw (0,9.0938) node {$2$};
		\draw (3,9.0938) node {$3$};
		\draw (4.5,9.0938) node {$2$};
		\draw (7.5,9.0938) node {$3$};
		\draw (9.1,9.0938) node {$2$};
		\draw (11.9,9.0938) node {$3$};
		
		\draw (0.75,10.3928) node {$1$};
		\draw (2.25,10.3928) node {$2$};
		\draw (5.25,10.3928) node {$1$};
		\draw (6.75,10.3928) node {$2$};
		\draw (9.75,10.3928) node {$1$};
		\draw (11.25,10.3928) node {$2$};

	\end{tikzpicture}
		\caption{Hexagonal grid \,\,\,  }
	\label{key3}
	\end{subfigure}
\begin{subfigure}[t]{0.48\textwidth}
	\centering

		\begin{tikzpicture}[line width=0.2pt,every paath/.style={gray!40}]

		\draw[Melon,very thin] (-2.3,0) -- (2.3,0);
		\draw[Melon,very thin] (-1.8,0.86) -- (1.8,0.86);
		\draw[Melon,very thin] (-1.8,-0.86) -- (1.8,-0.86);
		\draw[Melon,very thin] (-1.3,1.73) -- (1.3,1.73);
		\draw[Melon,very thin] (-1.3,-1.73) -- (1.3,-1.73);
		\draw[Melon,very thin] (-1.2,-2.07) -- (1.2,2.07);
		\draw[Melon,very thin] (-1.7,-1.2) -- (0.2,2.07);
		\draw[Melon,very thin] (1.7,1.2) -- (-0.2,-2.07);
		\draw[Melon,very thin] (-2.2,-0.34) -- (-0.8,2.07);
		\draw[Melon,very thin] (2.2,0.34) -- (0.8,-2.07);
		\draw[Melon,very thin] (1.2,-2.07) -- (-1.2,2.07);
		\draw[Melon,very thin] (-1.7,1.2) -- (0.2,-2.07);
		\draw[Melon,very thin] (1.7,-1.2) -- (-0.2,2.07);
		\draw[Melon,very thin] (-2.2,0.34) -- (-0.8,-2.07);
		\draw[Melon,very thin] (2.2,-0.34) -- (0.8,2.07);
		
		\draw[ Melon,very thin] (2.3,0) -- (2.5,0);
	\draw[ Melon,very thin] (-2.3,0) -- (-2.5,0);
	\draw[ Melon,very thin] (-1.8,0.86) -- (-2,0.86);
	\draw[Melon,very thin] (1.8,0.86) -- (2,0.86);
	\draw[Melon,very thin] (-1.8,-0.86) -- (-2,-0.86);
	\draw[ Melon,very thin] (1.8,-0.86) -- (2,-0.86);
	\draw[Melon,very thin] (1.3,1.73) -- (1.5,1.73);
	\draw[Melon,very thin] (1.3,-1.73) -- (1.5,-1.73);
	\draw[Melon,very thin] (-1.3,1.73) -- (-1.5,1.73);
	\draw[Melon,very thin] (-1.3,-1.73) -- (-1.5,-1.73);
	
	\draw[Melon,very thin] (-1.7,-1.2) -- (-1.8,-1.37);
	\draw[Melon,very thin] (0.2,2.07) -- (0.3,2.26);
	
	\draw[Melon,very thin] (-1.7,1.2) -- (-1.8,1.37);
	\draw[Melon,very thin] (0.2,-2.07) -- (0.3,-2.26);
	
	\draw[Melon,very thin] (1.7,-1.2) -- (1.8,-1.36);
	\draw[Melon,very thin] (-0.2,2.07) -- (-0.3,2.27);

	\draw[Melon,very thin] (1.7,1.2) -- (1.8,1.37);
	\draw[Melon,very thin] (-0.2,-2.07) -- (-0.3,-2.26);
	
	\draw[Melon,very thin] (-2.2,-0.34) -- (-2.3,-0.5);
	\draw[Melon,very thin] (1.2,2.07) -- (1.31,2.27);
	
	\draw[Melon,very thin] (2.2,-0.34) -- (2.3,-0.5);
	\draw[Melon,very thin] (-1.2,2.07) -- (-1.3,2.26);
	
	\draw[Melon,very thin] (-2.2,0.34) -- (-2.3,0.5);
	\draw[Melon,very thin] (1.2,-2.07) -- (1.3,-2.26);
	
	\draw[Melon,very thin] (2.2,0.34) -- (2.3,0.5);
	\draw[Melon,very thin] (-1.2,-2.07) -- (-1.3,-2.26);
	
	\draw[Melon,very thin] (0.8,2.07) -- (0.7,2.26);
	\draw[Melon,very thin] (-0.8,-2.07) -- (-0.7,-2.26);
	
	\draw[Melon,very thin] (0.8,-2.07) -- (0.7,-2.26);
	\draw[Melon,very thin] (-0.8,2.07) -- (-0.68,2.26);
		
		\draw (-1,-1.7) node {$1$};	
	\draw (0,-1.7) node {$2$};
	\draw (1,-1.7) node {$1$};
	
	\draw (-1.5,-0.9) node {$3$};
	\draw (-0.5,-0.9) node {$4$};	
	\draw (0.5,-0.9) node {$3$};
	\draw (1.5,-0.9) node {$4$};
	
	\draw (-2,0) node {$1$};
	\draw (-1,0) node {$2$};	
	\draw (0,0) node {$1$};
	\draw (1,0) node {$2$};
	\draw (2,0) node {$1$};
	
	\draw (-1.5,0.85) node {$3$};
	\draw (-0.5,0.85) node {$4$};	
	\draw (0.5,0.85) node {$3$};
	\draw (1.5,0.85) node {$4$};

	\draw (-1,1.75) node {$1$};	
	\draw (0,1.75) node {$2$};
	\draw (1,1.75) node {$1$};

		
	\end{tikzpicture}
	\caption{Triangular grid}
\label{key1}
\end{subfigure}
\caption{$\fd$-colorings of the hexagonal and triangular grids}
\label{fig:lattices}
\end{figure}
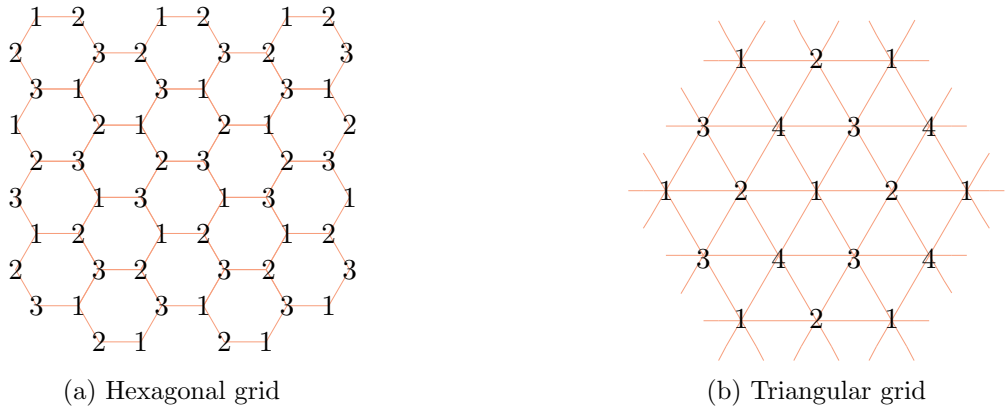

We show that several well-known infinite lattices enjoy the equality from the title of this section. In Figure~\ref{fig:grids1}, the infinite Cartesian grid $\mathbb{Z}\square\mathbb{Z}$ and the infinite King's grid $\mathbb{Z}\boxtimes\mathbb{Z}$ are depicted, along with their $2$-frugal colorings. Figure~\ref{fig:lattices} depicts the infinite hexagonal lattice $\cal H$ and the infinite triangular lattice $\cal T$ again with their $2$-frugal colorings. The colorings yield that each of the infinite graphs $G$ attains the value $\fd(G)=\lceil \Delta(G)/2\rceil+1.$ Notably, $\Delta(\mathbb{Z}\square\mathbb{Z})=4$, $\Delta(\mathbb{Z}\boxtimes\mathbb{Z})=8$, $\Delta({\cal H})=3$ and $\Delta({\cal T})=6$, whereas $\fd(\mathbb{Z}\square\mathbb{Z})=3$, $\fd(\mathbb{Z}\boxtimes\mathbb{Z})=5$, $\fd({\cal H})=3$ and $\fd({\cal T})=4$.

We continue with Cartesian powers of the two-way infinite path, denoted by $\mathbb{Z}^{\square,n}$. Note that it represents the (infinite) graph $\mathbb{Z}\,\square\cdots\square\,\mathbb{Z}$, where there are $n$ factors. In establishing that the Cartesian powers of the two-way infinite path attain the basic lower bound, combine the fact $$\Big\lceil\frac{\Delta(\mathbb{Z}^{\square,n})}{2}\Big\rceil+1=n+1$$
with the statement of the following result. 

\begin{theorem}\label{}
If $n$ is a positive integer, then $\fd(\mathbb{Z}^{\square,n})=n+1.$
\end{theorem}
\begin{proof}
Set $G=\mathbb{Z}^{\boxtimes,n}$ and $t=\lceil \Delta(G)/2\rceil+1=n+1$ . By the inequality (\ref{Delta}) with $k=2$, we have $\chi_{2}^{f}(G)\geq\lceil \Delta(G)/2\rceil+1$, so in the rest of the paper we consider the reverse inequality. 

 We define a coloring $f$ of the vertices of $G$ by $$f(x)=\sum_{i=1}^{n}ix_{i} \pmod t$$ for each $x=(x_{1},\ldots,x_{n})\in V(G)$. Let $x=(x_{1},\ldots,x_{n})$ and $y=(y_{1},\ldots,y_{n})$ be adjacent vertices in $G$. By definition, $|x_{i}-y_{i}|=1$  for some $i\in[n]$ and $x_{j}=y_{j}$ and each $j\in[n]\setminus \{i\}$. Due to this, $f(x)-f(y)=i(x_{i}-y_{i})$ (mod $t$). Since $|x_{i}-y_{i}|=1$ and $i\in[n]$, it follows that $f(x)\not\equiv f(y)$ (mod $t$). Therefore, $f$ is a proper coloring of $G$.

Let $x=(x_{1},\ldots,x_{n})\in V(G)$ and  $f(x)=c\in\{0,\ldots,t-1\}$ (mod $t$). Let $y=(y_{1},\ldots,y_{n})$ differ from $x$ only in the $i$th coordinate. If $y_{i}=x_{i}-1$, then $f(x)-f(y)=i(x_{i}-y_{i})=i$ (mod $t$). Therefore, $f(y)=c-i$ (mod $t$). On the other hand, if $y_{i}=x_{i}+1$, then $f(y)=c+i$ (mod $t$). Thus, $f$ assigns $c-i$ or $c+i$ (taken modulo $t$) to any neighbor of $x$ that differs from $x$ in the $i$th coordinate.
Suppose to the contrary that there exists a vertex $x=(x_{1},\ldots,x_{n})$ adjacent to three distinct vertices $w=(w_{1},\ldots,w_{n})$, $y=(y_{1},\ldots,y_{n})$ and $z=(z_{1},\ldots,z_{n})$ in $G$ such that $f(w)=f(y)=f(z)$ (mod $t$). With the above argument in mind, we may assume without loss of generality that $f(y)=c-i$ (mod $t$) and $f(z)=c-j$ (mod $t$) for some $i,j\in[n]$. Therefore, $i-j\equiv0$ (mod $t$). This necessarily implies that $i=j$, and hence, $y_{i}=x_{i}-1=z_{i}$. This contradicts the fact that $y$ and $z$ are distinct. Thus, $f$ is a $2$-frugal coloring of $G$ with $t$ colors. 
\end{proof}

As an immediate consequence of the theorem above, we get the following result for $n$-dimensional grids.  
\begin{corollary}
\label{cor:shortpaths}
    If $G=\square_{i=1}^{n}P_{m_i}$ with $m_{i}\geq 3$ for all $i\in [n]$ and $n\geq 1$, then $\chi_{2}^{f}(G)=\lceil\frac{\Delta(G)}{2}\rceil+1$.
\end{corollary}
The requirement $m_i\ge 3$ for all $i$ is needed to maintain the same maximum degree as in the product of infinite paths. In fact, it can happen that the basic lower bound is not attained if some $m_i$ are equal to $2$. In particular, note that $\fd(Q_3)=4$, while $\Delta(Q_3)=3$, and so the basic lower bound equals $3$. 

In order to prove a similar result for the strong power of paths, we need the following elementary but useful lemma.

\begin{lemma}\label{lem:integers}
Letting $V=\{-1,0,1\}^{n}$ we define 
$$\eta(v)=\sum_{i=1}^{n}3^{i-1}v_i$$ 
for each $v=(v_1,\ldots,v_n)\in V$. The function $\eta:V\to I$, where $I=\{-\frac{3^{n}-1}{2},-\frac{3^{n}-1}{2}+1,\ldots,\frac{3^{n}-1}{2}\}$, is bijective. 
\end{lemma}
\begin{proof}
The proof that $-(3^{n}-1)/2\leq \eta(v)\leq(3^{n}-1)/2$ for all $x\in V$ is elementary. 

Let $v=(v_{1},\ldots,v_{n})$ and $y=(y_{1},\ldots,y_{n})$ be distinct $n$-tuples in $V$, and let $k$ be the largest index such that $v_k\neq y_k$. This implies that $v_i=y_i$ for all $i> k$. Therefore, $\eta(v)-\eta(y)=3^{k-1}(v_k-y_k)+\sum_{i=1}^{k-1}3^{i-1}(v_i-y_i)$. Since $|v_i-y_i|\le 2$ for all $i\in[n]$ and $v_k\ne y_k$, it follows that $|3^{k-1}(v_k-y_k)|\geq3^{k-1}$ and $|\sum_{i=1}^{k-1}3^{i-1}(v_i-y_i)|\leq 2\sum_{i=1}^{k-1}3^{i-1}=3^{k-1}-1$. Thus, $3^{k-1}(v_k-y_k)+\sum_{i=1}^{k-1}3^{i-1}(v_i-y_i)\neq0$, and so $\eta(v)\neq \eta(y)$.

Note that $|V|=3^n=|I|$, which implies that $\eta:V\to I$ is bijective.
\end{proof}

The {\em $n$th strong power} of a graph $G$, denoted by $G^{\boxtimes,n}$, is defined as
$G\,\boxtimes\cdots\boxtimes\,G$, where there are $n$ factors.  
In establishing that the strong powers of the two-way infinite path attain the basic lower bound, combine the fact $$\Big\lceil\frac{\Delta(\mathbb{Z}^{\boxtimes,n})}{2}\Big\rceil+1=\frac{3^{n}+1}{2}$$
with the statement of the following result. 
\begin{theorem}\label{thm:strongpowerofpath}
If $n$ is a positive integer, then $\chi_{2}^{f}(\mathbb{Z}^{\boxtimes,n})=\frac{3^{n}+1}{2}$.
\end{theorem}
\begin{proof}
Set $G=\mathbb{Z}^{\boxtimes,n}$.
By Observation \ref{ob:lowerboundfor2} with $t=2$, we have $\chi_{2}^{f}(G)\geq\lceil \Delta(G)/2\rceil+1=(3^n-1)/2+1$. Let us now prove the reverse inequality.

Set $t=(3^{n}+1)/2$, and define a coloring $c$ of the vertices of $G$ by $$c(x)=\sum_{i=1}^{n}3^{i-1}x_i\pmod t$$ 
for each $x=(x_1,\ldots, x_n)\in V(G)$. Let $x=(x_1,\ldots,x_n)$ and $y=(y_1,\ldots,y_n)$ be adjacent vertices in $G$. By the adjacency rule of the strong product graph $G$, we have $x_i-y_i\in\{-1,0,1\}$ for all $i\in[n]$, and $c(x)-c(y)=\sum_{i=1}^{n}3^{i-1}(x_i-y_i)\pmod t$. Note that the integer $\sum_{i=1}^{n}3^{i-1}(x_i-y_i)$ equals $\eta(x-y)$, where $\eta$ is the function in Lemma~\ref{lem:integers}, it is non-zero and has absolute value less than $t$. Hence, $c(x)\not\equiv c(y)\pmod t$, and therefore $c$ is a proper coloring of $G$.

Suppose to the contrary that there exists a vertex $x=(x_1,\ldots,x_n)$ adjacent to three distinct vertices $w=(w_1,\ldots,w_n)$, $y=(y_1,\ldots,y_n)$ and $z=(z_1,\ldots,z_n)$ in $G$ such that $c(w)\equiv c(y)\equiv c(z)\pmod t$. For adjacent vertices $u\in\{w,y,z\}$ and $x$, we have $c(u)=c(x)-\sum_{i=1}^{n}3^{i-1}(x_i-u_i)\pmod t$. 

Let $a_u=\sum_{i=1}^{n}3^{i-1}(x_i-u_i)$ for $u\in\{w,y,z\}$. Clearly, $c(u)\equiv c(v)\pmod t$ if and only if $a_u\equiv a_v \pmod t$ for any $u,v\in\{w,y,z\}$. 
Since $x-w$, $x-y$ and $x-z$ are non-zero distinct $n$-tuples in $V$, and $a_u=\eta(x-u)$ for all $u\in\{w,y,z\}$, Lemma \ref{lem:integers} implies that $a_w$, $a_y$ and $a_z$ are non-zero distinct integers. Thus, at least two integers in $\{a_w,a_y,a_z\}$ are of the same sign. So, we may sssume that $a_w$ and $a_y$ are both positive, and that $a_w<a_y$ in view of Lemma~\ref{lem:integers}. We now infer that $a_{w}\not\equiv a_{y}$ as $0<a_w<a_y\leq t$, a contradiction. Therefore, no vertex has three neighbors of the same color, showing $\chi_2^{f}(G)\le t$.
\end{proof}

We mention two other families of graphs that attain the basic lower bound. Recall that among torus graphs $C_m\square C_n$ this property holds if and only if both $m$ and $n$ are divisible by $3$ (Proposition~\ref{Torus}). 
Another class of graphs with the desired property are claw-free cubic graphs different from $K_4$; see Proposition~\ref{prp:clawfree}. It would be interesting to find a characterization of all cubic graphs $G$ with $\fd(G)=3$.

\section{Concluding remarks}

  We conclude the paper by posing some open problems that arise from our work. We start with a question about computational complexity. Recall that the decision problem associated with $\fd$ is known to be NP-complete. We think that even its restriction, which is determining whether the $2$-frugal chromatic number of a given graph attains the basic lower bound, might be difficult.

\begin{p}
Is determining whether a graph $G$ satisfies $\fd(G)=\lceil\Delta(G)/2\rceil+1$ an NP-complete problem?
\end{p}

The block graphs $G$ that achieve the trivial lower bound $\fd(G)\ge \chi(G)$ have a nice characterization: by Theorem~\ref{thm:block}, they are exactly those block graphs $G$ for which $\omega(G)\ge \lceil\Delta(G)/2\rceil+1$. We propose the challenge to find its extension from block graphs to all chordal graphs.  
\begin{p}
Characterize chordal graphs $G$ with $\fd(G)=\chi(G)$.
\end{p}

In Theorem~\ref{boundsDelta=3}, we proved that $\fd(G)\le 5$ when $G$ is a subcubic graph. Despite the involved proof of this result, we do not know whether the bound can be improved to $4$, and pose it as an open problem.     

\begin{p}
Is there a graph $G$ with $\Delta(G)=3$ and $\fd(G)=5$?
\end{p}

In Theorem~\ref{N-G} we proved the sharp lower bound 
\begin{equation}
\label{eq:equality}
    \chi_{2}^{f}(G)+\chi_{2}^{f}(\overline{G})\geq \dfrac{n}{2}+2, 
\end{equation}
    which holds for all graphs $G$ with the exception of six special graphs on $9$ vertices. In view of this, we pose the following problem. 
    
\begin{p}
Characterize graphs $G$ of order $n$ that achieve the bound in~\eqref{eq:equality}.
\end{p}

Note that for the Cartesian products of paths, each of which has at least three vertices the $2$-frugal chromatic number has been determined (see Corollary~\ref{cor:shortpaths}). However, if each of the paths is of length $1$, the value is not known. Hence, we conclude the paper with the following open problem:
\begin{p}
  Determine $\fk(Q_n)$ for hypercubes $Q_n$ and $t\ge 2$.
\end{p}    



\section*{Acknowledgments}

B.B. was supported by the Slovenian Research and Innovation agency (grants P1-0297, N1-0285, and N1-0431). 
W.H. was supported by the National Natural Science Foundation of China (No.\ 12371345).


\end{document}